\documentclass[11pt,reqno,a4paper]{amsart}

\usepackage[latin1]{inputenc}
\usepackage{tikz}
\usetikzlibrary{shapes,arrows}
\usetikzlibrary{arrows.meta}
\usepackage{forest}
\usepackage{tikz-qtree}

\usetikzlibrary{arrows,shapes,positioning,shadows,trees}

\usepackage[hidelinks]{hyperref}

\usepackage{etoolbox}
\usepackage{amsmath}
\usepackage{enumerate}
\usepackage{amssymb}
\usepackage{amscd}
\usepackage{amsthm}
\usepackage{amsfonts}
\usepackage{graphicx}
\usepackage[all,cmtip]{xy}
\usepackage{enumitem}

\patchcmd{\subsection}{-.5em}{.5em}{}{}
\patchcmd{\subsubsection}{-.5em}{.5em}{}{}

\usepackage{enumitem}

\makeatother
\usepackage{hyperref}

\numberwithin{equation}{section}

\newcommand{\cA}{\mathcal{A}}
\newcommand{\cB}{\mathcal{B}}
\newcommand{\cC}{\mathcal{C}}

\newcommand{\cH}{\mathcal{H}}

\newcommand{\cP}{\mathcal{P}}

\newcommand{\bE}{\mathbb{E}}

\newcommand{\bN}{\mathbb{N}}

\newcommand{\bR}{\mathbb{R}}

\newcommand{\ra}{\rightarrow}

\newcommand{\qand}{\quad \textrm{and} \quad}

\newcommand\subsetsim{\mathrel{%
\ooalign{\raise0.2ex\hbox{$\subset$}\cr\hidewidth\raise-0.8ex\hbox{\scalebox{0.9}{$\sim$}}\hidewidth\cr}}}
\newcommand{\eps}{\varepsilon}

\DeclareMathOperator{\Ent}{Ent}

\theoremstyle{theorem}
\newtheorem{theorem}{Theorem}[section]
\newtheorem{corollary}[theorem]{Corollary}
\newtheorem{proposition}[theorem]{Proposition}
\newtheorem{lemma}[theorem]{Lemma}

\theoremstyle{definition}
\newtheorem{definition}[theorem]{Definition}
\newtheorem{remark}[theorem]{Remark}

\tikzstyle{decision} = [diamond, draw, fill=blue!20, 
    text width=4.5em, text badly centered, node distance=3cm, inner sep=0pt]
\tikzstyle{block} = [rectangle, draw, fill=blue!20, 
    text width=5em, text centered, rounded corners, minimum height=2em]
\tikzstyle{line} = [draw, -latex']
\tikzstyle{cloud} = [draw, ellipse,fill=red!20, node distance=3cm,
    minimum height=2em]

\renewcommand\labelenumi{(\roman{enumi})}
\renewcommand\theenumi\labelenumi 

\newcommand{\textoverline}[1]{$\overline{\mbox{#1}}$}

\usepackage{changepage}
\usepackage{mathtools}
\usepackage{graphicx}
\usepackage{calc}

\DeclarePairedDelimiterX\Set[2]{\{}{\}}{#1\,\delimsize\vert\,#2}

\begin{document}
\bibliographystyle{plain} % Choose Phys. Rev. style for bibliography

\title{Kud\textoverline{o}-continuity of entropy functionals}

\author{Michael Bj\"orklund}

\address{Department of Mathematics, Chalmers, Gothenburg, Sweden}
\email{micbjo@chalmers.se}

\author{Yair Hartman}
\address{Department of Mathematics, Ben Gurion University of the Negev, Be'er-Sheva, Israel}
\email{hartmany@bgu.ac.il}

\author{Hanna Oppelmayer}
\address{Department of Mathematics, Chalmers, Gothenburg, Sweden}
\email{hannaop@chalmers.se}
\date{}

\thanks{MB was supported by GoCas Young Excellence grant 11423310 and Swedish VR-grant 11253320, YH was partially supported by ISF grant 1175/18.}



\begin{abstract}
We study in this paper real-valued functions on the space of all sub-$\sigma$-algebras of a probability measure space, and
introduce the notion of Kud\textoverline{o}-continuity, which is an a priori strengthening of continuity with respect to strong convergence. 
We show that a large class of entropy functionals are Kud\textoverline{o}-continuous. On the way, we establish upper and 
lower continuity of various entropy functions with respect to asymptotic second order stochastic domination, which should be of 
independent interest. An application to the study of entropy spectra of $\mu$-boundaries associated to random walks on locally 
compact groups is given. 
\end{abstract}

\maketitle

\section{Introduction}

\subsection{The rough goal of the paper}
The need to understand the asymptotic behaviour of conditional expectations with respect to \emph{non-monotone} sequences of 
sub-$\sigma$-algebras, arises in a plethora of research fields, ranging from stochastic optimization \cite{Art2,Art3,Art01} to mathematical 
economy \cite{All83-1,All83-2,Cott86,Cott87} to jump processes \cite{Ch} to entropy theory of random walks on groups \cite{Bow,BHT,BHT2} (see 
Section \ref{sec:Borelprel} below for a discussion about the latter connection). An extensive analytic theory for non-monotone sequences of $\sigma$-algebras has been developed in \cite{Alon, AlonBram, Art01,Bo71,Fett77,Kudo,Nev72,Picc98,Rogg74}. Contrary to what happens for \emph{monotone}
increasing (or decreasing) sequences of sub-$\sigma$-algebras, in which case classical martingale theory is applicable and  
shows that the join (or meet) of the $\sigma$-algebras involved is always a natural ``limit'' of the sequence, non-monotone sequences typically lack convergent sub-sequences with respect to strong convergence (see Subsection \ref{subsec:spaceinfo} for definitions). Various remedies for this 
sequential non-compactness have been suggested over the years (see e.g. \cite{Kudo, Art01,Picc98}). In this paper we shall follow the ideas of 
Kud\textoverline{o} in his very insightful paper \cite{Kudo}: under some mild assumptions on the underlying measure space, it is possible to associate with every sequence of sub-$\sigma$-algebras a \emph{minimal upper Kud\textoverline{o}-limit} and a \emph{maximal lower Kud\textoverline{o}-limit} (see Subsection \ref{subsec:spaceinfo} for definitions and further discussions). These $\sigma$-algebras are unique (up to null sets) and they coincide (modulo null sets) if and only if the sequence converges strongly. Furthermore, the minimal upper (maximal lower) Kud\textoverline{o}-limit is typically strictly smaller (larger) than the set-theoretical limsup (liminf) of the sequence of sub-$\sigma$-algebras.

We shall consider a class of convex functionals, known as entropy functionals, defined on the convex set of probability densities on a given probability measure space. Our aim is to understand the asymptotic behaviour of these functionals along conditional expectations of a fixed probability density with respect to a (not necessarily monotone) sequence of sub-$\sigma$-algebras. In particular we are interested in whether these functionals are lower and upper semi-continuous with respect to the notions of lower and upper Kud\textoverline{o}-limits of the sequence of sub-$\sigma$-algebras (see Subsection \ref{subsec:spaceinfo} for definitions). These questions naturally occur in the entropy theory of random walks on groups, and some applications are discussed in more details below. 

We believe however that the main ideas of this paper are best phrased in the language of asymptotic second order stochastic domination and averaged entropy, which is why the first part of this paper is devoted to developing an asymptotic theory for entropy functions in this setting. Only later shall we
connect this theory to the discussion about Kud\textoverline{o}-limits above.

\subsection{Entropy functions and asymptotic second order stochastic domination}

Let us now define our basic set up.  Throughout the rest of this section, let $(X,\cB_X)$ be a measurable
space, and let $\xi$ be a probability measure on $\cB_X$. We 
denote by $\cP$ the set of probability densities on $(X,\xi)$, that is to say, the set of all measurable functions $f : X \ra [0,\infty)$ with $\int_X f \, d\xi = 1$. \\

Given $f \in \cP$, we define
\[
\alpha_f(t) = \int_t^\infty \xi(\{ f \geq \tau \}) \, d\tau, \quad \textrm{for $t \geq 0$}.
\]
We recall the classical notion of second order stochastic domination: if $f_1, f_2 \in \cP$, then we say that $f_1$ is \emph{second-order stochastically dominated} by $f_2$ if $\alpha_{f_1}(t) \leq \alpha_{f_2}(t)$ for all $t \geq 0$. We suggest the following asymptotic extension of this notion. 

\begin{definition}[Upper and lower limits]
Let $(f_n)$ be a sequence in $\cP$ and let $f \in \cP$. 
\vspace{0.1cm}
\begin{itemize}
\item We say that $f$ is an \emph{upper limit} of $(f_n)$ if 
\[
\varlimsup_n \alpha_{f_n}(t) \leq \alpha_f(t), \quad \textrm{for all $t \geq 0$}.
\]
\item We say that $f$ is a \emph{lower limit} of $(f_n)$ if 
\[
\alpha_f(t) \leq \varliminf_n \alpha_{f_n}(t), \quad \textrm{for all $t \geq 0$}.
\]
\end{itemize}
\end{definition}

Let $\Phi : [0,\infty) \ra \bR$ be a convex function with $\Phi(1) = 0$, and define the \emph{$\Phi$-entropy functional} 
\[
\Ent^\Phi_\xi : \cP \ra [0,\infty], \quad \Ent_\xi^\Phi(f) = \int_X \Phi(f(x)) \, d\xi(x), \quad \textrm{for $f \in \cP$}.
\]
Note that Jensen's inequality, together with the assumption that $\Phi(1) = 0$, shows that $\Ent_\xi^\Phi(f) \geq 0$ for all $f \in \cP$.
We shall mostly work with the subset $\cP_\Phi \subset \cP$ consisting of those $f \in \cP$ which satisfy
\[
\int_X |\Phi(f)| \, d\xi < \infty \qand \lim_{t \ra \infty} \alpha_f(t) \, \Phi'(t) = 0.
\]
Throughout the rest of this introduction, we shall assume that 
\vspace{0.1cm}
\begin{enumerate}
\item[(i)] $\Phi$ is twice continuously differentiable on $(0,\infty)$ and $\lim_{t \ra 0^+} t \, \Phi'(t) = 0$.
\vspace{0.1cm}
\item[(ii)] $\Phi(0) = 0$ and there exists a point $0 < t_o < 1$ such that $\Phi'$ is strictly negative on $(0,t_o)$ and strictly positive on $(t_o,\infty)$.
\end{enumerate}
These assumptions are clearly satisfied by the function 
\begin{equation}
\label{standard}
\Phi(t) = 
\left\{
\begin{array}{cc}
0 & \textrm{for $t = 0$} \\[5pt]
t \log t & \textrm{for $t > 0$}
\end{array}
\right.,
\end{equation}
with $t_o = e^{-1}$. The corresponding $\Phi$-entropy functional is usually denoted by $\Ent_\xi$, and is often referred to as the 
\emph{standard entropy functional} on $\cP$. The literature devoted to the study of this functional is vast, and we simply refer the 
reader to Section \cite[Section 5]{Bak} or \cite[Section 16]{Sim} for nice surveys, along with extensive references.  \\

Our main result concerning upper/lower limits and $\Phi$-entropy functionals can now be formulated as follows. It will be proved in 
Section \ref{sec:proofUppLow}.
\begin{theorem}
\label{thm_upper/lower-limits}
Let $(f_n)$ be a sequence in $\cP_\Phi$ and let $f$ be an element of $\cP_\Phi$.
\vspace{0.2cm}
\begin{enumerate}
\item[(i)] Suppose that $f$ is a lower limit of $(f_n)$. Then,
\[
\Ent_\xi^{\Phi}(f) \leq \varliminf_n \Ent_\xi^{\Phi}(f_n).
\]
\vspace{0.05cm}
\item[(ii)] Suppose that $f$ is an upper limit of $(f_n)$, and that for some $\beta > 0$,
\[
\int_1^\infty t^{-\beta} \, \Phi''(t) \, dt < \infty \qand \sup_n \int_X f_n^{1+\beta} \, d\xi < \infty.
\]
Then,
\[
\varlimsup_n \Ent_\xi^{\Phi}(f_n) \leq \Ent_\xi^{\Phi}(f).
\]
\end{enumerate}
\end{theorem}

\begin{remark}
It is an easy exercise to show that the additional assumptions on $\Phi$ in (ii) are satisfied for the standard entropy functional, where $\Phi$ is given by \eqref{standard}.
\end{remark}

Roughly speaking, Theorem \ref{thm_upper/lower-limits} asserts that $\Ent_\xi^\Phi$ is both upper and lower semi-continuous with 
respect to the notions of upper and lower limits in $\cP$. In particular, we have the following corollary.

\begin{corollary}
Let $(f_n)$ be a sequence in $\cP_\Phi$ and let $f$ be an element of $\cP_\Phi$ which is both a lower and an upper limit
of the sequence $(f_n)$. If for some $\beta > 0$,
\[
\int_1^\infty t^{-\beta} \, \Phi''(t) \, dt < \infty \qand \sup_n \int_X f_n^{1+\beta} \, d\xi < \infty,
\]
then $\lim_n \Ent^\Phi_\xi(f_n) = \Ent_\xi^\Phi(f)$.
\end{corollary}

\subsection{Entropy functionals and their Kud\textoverline{o}-continuity}

Let us now connect upper and lower limits to our discussion about Kud\textoverline{o}-limits above. We give the following 
definition here, and refer to Subsection \ref{subsec:spaceinfo} for a more detailed discussion about the notions involved. 
For simplicity, we shall assume that $L^1(X,\xi)$ is separable in the norm topology.

\begin{definition}[Upper and lower Kud\textoverline{o}-limits]
Let $(\cA_n)$ be a sequence of sub-$\sigma$-algebras of $\cB_X$ and let $\cA$ be a sub-$\sigma$-algebra of $\cB_X$. We say 
that 
\vspace{0.2cm}
\begin{itemize}
\item $(\cA_n)$ \emph{converges strongly} to $\cA$ if 
\[
\lim_n \|\bE_{\cA}(f) - \bE_{\cA_n}(f)\|_{1} = 0, \quad \textrm{for every $f \in L^1(X,\xi)$}.
\]
\item $\cA$ is an \emph{upper Kud\textoverline{o}-limit} of $(\cA_n)$ if
\[
\varlimsup_n \|\bE_{\cA_n}(f)\|_1 \leq \|\bE_{\cA}(f)\|_1, \quad \textrm{for every $f \in L^1(X,\xi)$}.
\] 
\item $\cA$ is a \emph{lower Kud\textoverline{o}-limit} of $(\cA_n)$ if
\[
\|\bE_{\cA}(f)\|_1 \leq \varliminf_n \|\bE_{\cA_n}(f)\|_1, \quad \textrm{for every $f \in L^1(X,\xi)$}.
\] 
\end{itemize}
\end{definition}

Given a sequence $(\cA_n)$ of sub-$\sigma$-algebras, Kud\textoverline{o} (\cite[Theorem 3.3]{Kudo} and  \cite[Theorem 3.2]{Kudo}) shows
that there are always a \emph{minimal upper Kud\textoverline{o}-limit} $\cA^+$ and a \emph{maximal lower Kud\textoverline{o}-limit} $\cA^{-}$ of the sequence. Furthermore, $(\cA_n)$ converges strongly to $\cA^+$ if and only if $\cA^+$ and $\cA^{-}$ coincide modulo $\xi$-null sets \cite[Theorem 2.1 (ii)]{Kudo}. \\

In what follows, we let $\frak{S}(X,\xi)$ denote the space of information, that is to say, the set of all sub-$\sigma$-algebras of $\cB_X$,
identified up to $\xi$-null sets (see also Subsection \ref{subsec:spaceinfo} for more details). 

\begin{definition}[Kud\textoverline{o}-continuity]
A function
$F : \frak{S}(X,\xi) \ra [0,\infty]$ is \emph{Kud\textoverline{o}-continuous} if for every sequence $(\cA_n)$ in $\frak{S}(X,\xi)$,
we have
\[
F(\cA^{-}) \leq \varliminf_n F(\cA_n) \qand \varlimsup_n F(\cA_n) \leq F(\cA^+),
\]
where $\cA^+$ and $\cA^-$ denote the \emph{minimal upper Kud\textoverline{o}-limit} and the \emph{maximal upper Kud\textoverline{o}-limit}
of the sequence $(\cA_n)$.
\end{definition}

\begin{remark}
Since $(\cA_n)$ converges strongly to $\cA^+$ if and only if $\cA^+ \sim_\xi \cA^{-}$, we see that every Kud\textoverline{o}-continuous function
is also continuous with respect to strong convergence. We do not know to which extent the converse holds. 
\end{remark}

We relate in Corollary \ref{cor_relateKudotoAsymp2nd} below, the notions of upper and lower limits to upper and lower Kud\textoverline{o}-limits.
More precisely, we show that if $(\cA_n)$ is a sequence of sub-$\sigma$-algebras of $\cB_X$ and $\cA^{+}$ and $\cA^{-}$ are sub-$\sigma$-algebras
of $\cB_X$, then 
\begin{itemize}
\item $\cA^{+}$ is an upper Kud\textoverline{o}-limit of $(\cA_n)$ if and only if for every non-negative $\rho$ on $X$ with $\int_X \rho \, d\xi = 1$, 
the function $\bE_{\cA^{+}}(\rho)$ is an upper limit of the sequence $(\bE_{\cA_n}(\rho))$.
\vspace{0,2cm}
\item $\cA^{-}$ is an upper Kud\textoverline{o}-limit of $(\cA_n)$ if and only if for every non-negative $\rho$ on $X$ with $\int_X \rho \, d\xi = 1$, 
the function $\bE_{\cA^{-}}(\rho)$ is a lower limit of the sequence $(\bE_{\cA_n}(\rho))$.
\end{itemize}

Let $\Phi$ be as in the previous subsection. Given $\rho \in \cP_\Phi$, we define the \emph{$(\Phi,\rho)$-entropy} $H^\Phi_\rho(\cA)$ of 
a sub-$\sigma$-algebra $\cA \subset \cB_X$ by
\[
H^\Phi_\rho(\cA) = \Ent^\Phi_\xi(\bE_\cA(\rho)).
\]
It readily follows from Jensen's inequality that the map $\cA \mapsto H_\rho^{\Phi}(\cA)$ is increasing with respect to inclusions of $\sigma$-algebras. 
In what follows, we shall establish monotonicity with respect to upper and lower Kud\textoverline{o}-limits. 
The following theorem is a straightforward consequence of Theorem \ref{thm_upper/lower-limits}. We will provide the details of the proof in 
Section \ref{sec:proofofentropycont}.

\begin{theorem}
\label{thm_entropycont}
Let $\rho \in \cP_\Phi$ and let $(\cA_n)$ be a sequence of sub-$\sigma$-algebras of $\cB_X$. 
Let $\cA$ be a sub-$\sigma$-algebra of $\cB_X$. 
\begin{enumerate}
\item[(i)] Suppose that $\cA$ is a lower limit of $(\cA_n)$. Then,
\[
H^\Phi_\rho(\cA) \leq \varliminf_{n} H^\Phi_\rho(\cA_n).
\]
\item[(ii)] Suppose that $\cA$ is a lower limit of $(\cA_n)$, and that for some $\beta > 0$,
\[
\int_1^\infty t^{-\beta} \, \Phi''(t) \, dt < \infty \qand \int_X \rho^{1+\beta} \, d\xi < \infty.
\]
Then, 
\[
\varlimsup_{n} H^\Phi_\rho(\cA_n) \leq H^\Phi_\rho(\cA).
\]
\end{enumerate}
\end{theorem}

\begin{remark}
One might believe that the lower semi-continuity-assertion in (i), at least in the case when $\Phi$ is given by \eqref{standard}, is a straightforward consequence of 
the entropic inequality (see e.g. \cite[Equation 5.1.3]{Bak}), which in particular implies that the map $f \mapsto \Ent_\xi(f)$ is lower semi-continuous
with respect to the weak topology on $\cP \subset L^1(X,\xi)$. Hence, if $(\cA_n)$ is a sequence of sub-$\sigma$-algebras of $\cB_X$, then for 
every $\rho \in \cP$, the family $(\bE_{\cA_n}(\rho))$ is weakly pre-compact by the Dunford-Pettis Theorem. Since we assume that $L^1(X,\xi)$ is 
separable, a diagonal argument shows that we can construct a non-negative operator $P : L^1(X,\xi) \ra L^1(X,\xi)$ with $P1 = 1$ such that the 
sequence $(\bE_{\cA_{n_k}})$ converges in the weak operator topology to $P$ along some sub-sequence $(n_k)$, whence 
\[
\varliminf_k \Ent_\xi(\bE_{\cA_{n_k}}(\rho)) \geq \Ent_\xi(P\rho), \quad \textrm{for all $\rho \in L^1(X,\xi)$}.
\]
If $P$ were the conditional expectation with respect to a lower Kud\textoverline{o}-limit of the sequence $(\cA_n)$, then (i) in Theorem \ref{thm_entropycont} would indeed follow. 
However, as is demonstrated by the examples \cite[Example 3.1]{Kudo}  and \cite[Example 2.1]{Art01}, the weak operator topology accumulation points of the 
sequence $(\bE_{\cA_n})$ might not contain any conditional expectation operators.
\end{remark}

The following corollary is immediate.

\begin{corollary}
\label{cor_kudocont}
Let $\rho \in \cP_\Phi$ and suppose that for some $\beta > 0$,
\[
 \int_1^\infty t^{-\beta} \, \Phi''(t) \, dt < \infty \qand \int_X \rho^{1+\beta} \, d\xi < \infty.
\]
Then the map $\cA \mapsto H^\Phi_\rho(\cA)$ is Kud\textoverline{o}-continuous, whence also continuous with respect to strong convergence.
\end{corollary}

\begin{remark}
We stress that continuity of the map $\cA \mapsto H^\Phi_\rho(\cA)$ with respect to strong convergence can be established in simpler ways. 
\end{remark}

\subsection{Quantifying strong convergence in terms of entropy}

Let us now specialize our discussion to the standard entropy functional $\Ent_\xi$, associated to the convex function $\Phi$ given by \eqref{standard}.
Given $\rho \in \cP_\Phi$, we set
\[
H_\rho(\cA) = \Ent_\xi(\bE_{\cA}(\rho)), \quad \textrm{for $\cA \in \frak{S}(X,\xi)$}.
\]
We know by Corollary \ref{cor_kudocont} that $\cA \mapsto H_\rho(\cA)$ is continuous with respect to strong convergence. The following theorem
provides a kind of converse to this continuity: If $H_\rho(\cA_n) \ra H_\rho(\cA^+)$, where $\cA^+$ is an upper Kud\textoverline{o}-limit of $(\cA_n)$, then 
$\bE_{\cA_n}(\rho)$ converges to $\bE_{\cA^+}(\rho)$ in the $L^1$-norm.

\begin{theorem}
\label{thm_quantentropycont}
Let $\lambda \geq 1$ and suppose that $\rho : X \ra [\lambda^{-1},\lambda]$ is a measurable function such that $\int_X \rho \, d\xi = 1$.
Then, for every sequence $(\cA_n)$ of
sub-$\sigma$-algebras of $\cB_X$, we have
\begin{eqnarray*}
\varlimsup_n \, \| \bE_{\cA^{+}}(\rho) - \bE_{\cA_n}(\rho) \|_1 
&\leq & \sqrt{2} \Big( H_\rho(\cA^{+})- \varliminf_n H_\rho(\cA_n)\Big)^{1/2} \\
&\leq &  \sqrt{2} \Big( H_\rho(\cA^{+})- H_\rho(\cA^{-})\Big)^{1/2},
\end{eqnarray*}
where $\cA^{+}$ and $\cA^{-}$ denote upper and lower Kud\textoverline{o}-limits of $(\cA_n)$ respectively.
\end{theorem}

\begin{remark}
The last inequality in Theorem \ref{thm_quantentropycont} follows from (i) of Theorem \ref{thm_entropycont}.
\end{remark}

\subsection{Averaged entropy}\label{subsec:av}

We now arrive at the second main theme of the paper, namely the notion of \emph{averaged entropy}. To define it, we need a measure
space $(T,\eta)$ and a measurable kernel $K : T \times X \ra [0,\infty)$ which satisfies 
\vspace{0.2cm}
\begin{itemize}
\item for every $t \in T$, the function $\rho_t = K(t,\cdot)$ on $X$ is measurable, satisfies 
\[
\int_X \rho_t \, d\xi = 1 \qand \int_T \Ent_\xi(\rho_t) \, d\eta(t) < \infty,
\]
and there exists $\lambda_t \geq 1$ such that $\rho_t$ only takes values in $[\lambda_t^{-1},\lambda_t]$. \\
\vspace{0.2cm}
\item The bounded linear map $\psi \mapsto f_\psi$ from $L^\infty(T,\eta)$ to $L^1(X,\xi)$, defined by
\[
f_\psi = \int_T \psi(t) \, \rho_t \, d\eta(t), \quad \textrm{for $\psi \in L^\infty(T,\eta)$}
\]
has a norm-dense image $V \subset L^1(X,\xi)$.
\end{itemize}
In Subsection \ref{subsec:bndtheory} below we shall see that Poisson boundaries of random walks on locally compact groups provide a large
collection of examples of $(T,\eta)$, $(X,\xi)$ and $K$ for which these conditions hold. \\

Given $(T,\eta)$ and a kernel $K$ as above, we define the \emph{$\eta$-averaged entropy} $h_\eta(\cA)$ of a sub-$\sigma$-algebra $\cA \subset \cB_X$ by 
\[
h_\eta(\cA) = \int_T H_{\rho_t}(\cA) \, d\eta(t).
\]
Again we obtain continuity for this notion of entropy.
\begin{theorem}\label{thm_Tcont}
Let $(T,\eta)$ be as above and $\cA_n$ be a convergent sequence in $\frak{S}(X,\xi)$ with limit $\cA\in \frak{S}(X,\xi)$, then $$h_{\eta}(\cA_n)\to h_{\eta}(\cA).$$
\end{theorem}
Note that the assumptions on the kernel allow us to apply Corollary~\ref{cor_kudocont}, which already proves the above theorem.\\

The following question will occupy us from now. Let $(\cA_n)$ be a sequence of sub-$\sigma$-algebras of $\cB_X$ and let $\cA$ be a sub-$\sigma$-algebra of $\cB_X$. Suppose that $h_\eta(\cA_n) \ra h_\eta(\cA)$. Does this imply that $\cA_n \ra \cA$ strongly?  In other words, can we determine whether $(\cA_n)$ converges strongly to $\cA$ by considering whether or not the sequence 
$(h_\eta(\cA_n))$ of real numbers converges to $h_\eta(\cA)$?  \\

Our next theorem provides an affirmative answer in the cases
when $\cA$ is either the trivial $\sigma$-algebra (modulo null sets) or an upper Kud\textoverline{o}-limit of 
$(\cA_n)$. We shall prove it in Section \ref{sec:kernelconv}.

\begin{theorem}
\label{thm_kernelconv}
Let $(\cA_n)$ be a sequence of sub-$\sigma$-algebras of $\cB_X$. 
\vspace{0.2cm}
\begin{enumerate}
\item[(i)] Suppose that $h_\eta(\cA_n) \ra 0$. Then the sequence $(\cA_n)$ converges strongly to the trivial $\sigma$-algebra (modulo $\xi$-null sets).
\vspace{0.2cm}
\item[(ii)] Let $\cA$ be an upper Kud\textoverline{o}-limit of $(\cA_n)$. If $h_\eta(\cA_n) \ra h_\eta(\cA)$, then $\cA_n \ra \cA$ strongly.
\end{enumerate}
\end{theorem}

\subsection{An application to the entropy spectrum of random walks on groups}

Let us now discuss the main application in this paper of the theory developed in the previous theorems. Let $G$ be a locally compact, second countable and compactly generated 
group, and fix a \emph{left}-invariant Haar measure $m_G$ on $G$. Define the probability measure $d\mu = u \, dm_G$ on $G$, where we assume that $u$ is a continuous function on $G$ with relatively compact support such that \begin{equation}\label{generating}G = \bigcup_{k \geq 1} \big\{ u^{*k} > 0 \big\}.\end{equation} We shall refer to $(G,\mu)$ as a measured group. 
Given a $\mu$-stationary Borel $(G,\mu)$-space (we refer the reader to Section \ref{sec:Borelprel} for definitions and notation), we define its \emph{$\mu$-entropy} (first introduced by Furstenberg in \cite{Fur1}) by
\begin{equation}
\label{def_FuEnt}
h_{(G,\mu)}(X,\xi) = \int_G \Big( -\int_X \log \frac{dg^{-1}\xi}{d\xi} \, d\xi \Big) \, d\mu(g).
\end{equation}
The reader might recognize the $\mu$-entropy of $(X,\xi)$ as the $\mu$-integral of the Kullback-Leibler divergence between the measures $\xi$ and $g^{-1}\xi$. \\

Starting with the works \cite{NZ1,NZ2,N} by Nevo and Zimmer, the topological and categorical structure of the image of the map
\[
(X,\xi) \mapsto h_{(G,\mu)}(X,\xi),
\] 
as $(X,\xi)$ ranges over various sub-classes of $\mu$-stationary Borel $G$-spaces, have been subject to intense studies (see e.g. \cite{BHT,HY,TZ}). 
In this paper we shall be concerned with the restriction of this map to the class of $\mu$-boundaries (we refer the reader to Section \ref{sec:Borelprel} for the necessary definitions and notation). The \emph{Poisson boundary} of the measured group $(G,\mu)$, that we denote by $(B,\nu)$ is the maximal $\mu$-boundary, in the sense 
that all other $\mu$-boundaries are $G$-equivariant images of $(B,\nu)$. The one-point space (on which $G$ acts trivially) is the minimal $\mu$-boundary, and is called the \emph{trivial $\mu$-boundary}. \\

In order to state our main results here, we need to be able to talk about limits of $\mu$-boundaries. By this we shall mean the following. 
Let $((Z_n,\theta_n))$ be a sequence of $\mu$-boundaries, together with measurable and $G$-equivariant maps $\pi_n : (B,\nu) \ra (Z_n,\theta_n)$. Consider the sequence 
$(\cA_n)$ of $G$-invariant sub-$\sigma$-algebras defined by $\cA_n = \pi_n^{-1}(\cB_{Z_n})$ (where $\cB_{Z_n}$, as usual, is the sigma algebra implicit of  $(Z_n,\theta_n)$). We denote by $\cA^{+}$ and $\cA^{-}$ the minimal upper 
and maximial lower Kud\textoverline{o}-limit of the sequence $(\cA_n)$ respectively. We prove in Lemma \ref{Kudoinvariant} that both $\cA^+$ and $\cA^-$ 
are $G$-invariant, whence there are $\mu$-boundaries $(Z^+,\theta^+)$ and $(Z^-,\theta^-)$, together with measurable and $G$-equivariant maps
\[
\pi_{\pm} : (B,\nu) \ra (Z^{\pm},\theta^{\pm})
\] 
such that $\pi_{\pm}^{-1}(\cB_{Z^{\pm}}) = \cA^{\pm}$ (modulo $\nu$-null sets). We shall refer to 
$(Z^+,\theta^+)$ and $(Z^-,\theta^-)$ as the \emph{minimal upper and maximal upper Kud\textoverline{o}-limits of $((Z_n,\theta_n))$} respectively.
In particular, we say that $((Z_n,\theta_n))$ \emph{converges strongly} to $(Z^+,\theta^+)$ if $(\cA_n)$ converges strongly to $(\cA^+)$, or equivalently,
if $\cA^+ = \cA^-$ (modulo $\nu$-null sets). \\

We say that 
$(G,\mu)$
\textit{has bounded Radon-Nikodym derivatives}, if for every $g\in G$ there exists a constant $\lambda_g>1$ such that $\rho_g^{\nu}(b)=\frac{d g\nu}{d \nu}(b)\in[\lambda_g^{-1},\lambda_g]$ for $\nu$-a.e.\ $b\in B$. It is well known that any discrete group with  generating measure (see eq.~\ref{generating}) has bounded Radon-Nikodym derivatives. In Lemma~\ref{lemma_RNbounded} we will see that in fact a large class of measured groups has bounded Radon-Nikodym derivatives.
Bounded Radon-Nikodym derivatives satisfy the conditions of being a kernel in the sense of Subsection~\ref{subsec:av} (see Section~\ref{sec:FEnt}). Thus we can apply Theorem~\ref{thm_Tcont}  and obtain a continuity statement as follows.
\begin{theorem}\label{thm_bndcont} Let $(G,\mu)$ be a measured group with  bounded Radon-Nikodym derivatives. Then for every  sequence $((Z_n,\theta_n))$ of $\mu$-boundaries which converges strongly to a $\mu$-boundary $(Z,\theta)$, it holds that 
$$h_{(G,\mu)}(Z_n,\theta_n)\to h_{(G,\mu)}(Z,\theta),$$ i.e.\ $h_{(G,\mu)}$ is continuous w.r.t.\ the strong topology.
\end{theorem}

It is known that for groups with Kazhdan's property (T), there is a gap in the $\mu$-entropy (see e.g.\cite{N,BHT}) hence one concludes the following:
\begin{corollary}\label{cor_propT} Let $(G,\mu)$ be a measured group with bounded Radon-Nikodym derivatives. If $G$ has Kazhdan's property (T), then the trivial $\mu$-boundary is isolated in the topology of strongly convergent $\mu$-boundaries.
\end{corollary}

It is natural  to ask about the converse direction of Theorem~\ref{thm_bndcont} which leads to our next result. See Section~\ref{sec:FEnt} for a proof using Theorem~\ref{thm_kernelconv}. 

\begin{theorem}
\label{thm_FurstenbergEntropy} Given $(G,\mu)$ as above with bounded Radon-Nikodym derivatives.
Let $((Z_n,\theta_n))$ be a sequence of $\mu$-boundaries. 
\vspace{0.2cm}
\begin{enumerate}
\item[(i)] Suppose that $h_{(G,\mu)}(Z_n,\theta_n) \ra 0$. Then $((Z_n,\theta_n))$ converges strongly to the trivial $\mu$-boundary. 
\vspace{0.1cm}
\item[(ii)] Let $(Z^+,\theta^+)$ be an upper Kud\textoverline{o}-limit of $((Z_n,\theta_n))$. If $h_{(G,\mu)}(Z_n,\theta_n) \ra h_{(G,\mu)}(Z^+,\theta^+)$, 
then the sequence $((Z_n,\theta_n))$ converges strongly to $(Z^+,\theta^+)$.
\end{enumerate}
In particular, if a sequence of $\mu$-boundaries $((Z_n,\theta_n))$ satisfies that $h_{(G,\mu)}(Z_n,\theta_n) \ra h_{(G,\mu)}(B,\nu)$, then the sequence $((Z_n,\theta_n))$ converges strongly to the Poisson boundary $(B,\nu)$. 
\end{theorem}

The main point in this theorem is that in some cases we can establish strong convergence of a sequence of $\mu$-boundaries by merely proving convergence of their $\mu$-entropies.

\section{Preliminaries}

\subsection{The space of information and Kud\textoverline{o}-continuity}
\label{subsec:spaceinfo}
Let $(X,\cB_X)$ be a measurable space. We denote by $\frak{S}(X)$ the set of all sub-$\sigma$-algebras of $\cB_X$. Given a probability
measure $\xi$ on $\cB_X$, we endow $\frak{S}(X)$ with the following equivalence relation:  two elements $\cA$ and $\cA'$ are
\emph{$\xi$-equivalent}, written $\cA \sim_\xi \cA'$, if they differ only by $\xi$-null sets. In other words, $\cA \sim_\xi \cA'$ if for every
$A \in \cA$ and for every $A' \in \cA'$, there exist $B' \in \cA'$ and $B \in \cA$ such that 
\[
\xi(A \Delta B') = 0 \qand \xi(A' \Delta B) = 0.
\] 
Since the symmetric differences $A \Delta B'$ and $A' \Delta B$ both belong to $\cB_X$, this equivalence relation does not require $\xi$ to be a 
complete measure. 
The quotient space $\frak{S}(X,\xi) := 
\frak{S}(X)/\sim_\xi$ is sometimes referred to as \emph{the space of information of $(X,\cB_X,\xi)$} in the literature (see e.g. \cite{Cott86}). Furthermore, the choice of a measure $\xi$ on $\cB_X$ provides a map 
\[
\bE : \mathfrak{S}(X) \ra B(L^1(X,\xi)), \enskip \cA \mapsto \bE_{\cA},
\]
where $B(L^1(X,\xi))$ is the set of all bounded linear maps on $L^1(X,\xi)$ and $\bE_{\cA}$ is the conditional expectation with respect to 
the sub-$\sigma$-algebra $\cA$, relative to $\xi$. As is well-known (see e.g. \cite[Theorem 2]{Bo71}), we have $\bE_{\cA} = \bE_{\cA'}$
if and only if $\cA \sim_\xi \cA'$, whence $\bE$ descends to a map $\frak{S}(X,\xi) \ra B(L^1(X,\xi))$, which we still denote by $\bE$.

We shall always assume that $L^1(X,\xi)$ is separable. This assumption allows us to put a sequential topology on $\frak{S}(X,\xi)$ as follows. 
Let $(\cA_n)$ be a sequence in $\frak{S}(X,\xi)$ and let $\cA$ be an element of $\frak{S}(X,\xi)$ (we shall, somewhat abusively, denote elements in $\frak{S}(X)$ and $\frak{S}(X,\xi)$ by the same letters). The sequence $(\cA_n)$ is said to \emph{converge to $\cA$} if for every $f \in L^1(X,\xi)$, 
we have $\bE_{\cA_n}(f) \ra \bE_{\cA}(f)$ in the norm-topology on $L^1(X,\xi)$. The reader might recognize this as the pull-back to $\frak{S}(X,\xi)$, 
under the map $\bE$, of the strong operator topology on $B(L^1(X,\xi))$. Various aspects of this topology have been studied in for instance
\cite{Bo71,Nev72, Rogg74, Fett77, All83-1, All83-2, Cott86, Cott87, Alon, AlonBram, Art01, Kudo}.

Unfortunately, the space $\frak{S}(X,\xi)$, endowed with this topology is not compact (unless $\xi$ is rather special, e.g. purely atomic), and several 
explicit examples of sequences in $\frak{S}(X,\xi)$ with no convergent sub-sequences have been given in the literature; see e.g. \cite[Example 3.1]{Kudo}  and \cite[Example 2.1]{Art01}. In the very insightful paper \cite{Kudo}, Kud\textoverline{o} suggests a remedy for the non-compactness 
of the space of information along the following lines. He first observes \cite[Theorem 3.1]{Kudo} that if $\cA$ and $\cA'$ are sub-$\sigma$-algebras
of $\cB_X$, then $\cA \subset \cA'$ (modulo $\xi$-null sets) if and only if 
\[
\|\bE_{\cA}(f)\|_{L^1(\xi)} \leq \|\bE_{\cA'}(f)\|_{L^1(\xi)}, \quad \textrm{for all $f \in L^1(X,\xi)$}.
\]
Motivated by this equivalence, he goes on to prove that if $(\cA_n)$ is a sequence of sub-$\sigma$-algebras, then the set
\[
\Sigma^+ = \big\{ \cA \in \frak{S}(X,\xi) \, \mid \, \varlimsup_n \|\bE_{\cA_n}(f)\|_{L^1(\xi)} \leq \|\bE_{\cA}(f)\|_{L^1(\xi)}, \enskip \textrm{for all 
$f \in L^1(X,\xi)$} \big\}
\]
is closed under intersections \cite[Lemma 3.2]{Kudo}, whence by a straightforward application of Zorn's Lemma, there is a unique minimal element $\cA^+$ 
in $\Sigma^{+}$ \cite[Theorem 3.3]{Kudo}. Kud\textoverline{o} also proves that the set 
\[
\Sigma^- = \big\{ \cA \in \frak{S}(X,\xi) \, \mid \, \|\bE_{\cA}(f)\|_{L^1(\xi)} \leq \varliminf_n \|\bE_{\cA_n}(f)\|_{L^1(\xi)}, \enskip \textrm{for all 
$f \in L^1(X,\xi)$} \big\}
\]
has a unique \emph{maximal} element $\cA^{-}$, which, modulo $\xi$-null sets, is given by \cite[Theorem 3.2]{Kudo}
\[
\cA^{-} = \big\{ B \in \cB_X \, \mid \, \textrm{there exists $B_n \in \cA_n$ such that $\lim_n \xi(B \Delta B_n) = 0$} \big\}.
\]
Finally, Kud\textoverline{o} shows that the sequence $(\cA_n)$ converges to $\cA^+$ in the sense described above if and only if $\cA^+ \sim_\xi \cA^{-}$ \cite[Theorem 2.1 (ii)]{Kudo}. With these observations at hand, it is natural to introduce the following
definition.

\begin{definition}[Upper and lower Kud\textoverline{o}-limits]
The elements $\Sigma^{+}$ and $\Sigma^{-}$ are called the \emph{upper and lower Kud\textoverline{o}-limits} of the sequence $(\cA_n)$ respectively. 
We refer to $\cA^+$ and $\cA^{-}$ as \emph{the minimal upper and the maximal lower Kud\textoverline{o}-limit} of $(\cA_n)$ respectively.  
\end{definition}

Given the work of Kud\textoverline{o}, the following strengthening of the notion of continuity of functions on $\frak{S}(X,\xi)$ seems natural.

\begin{definition}[Kud\textoverline{o}-continuity]
We say that a function $F : \frak{S}(X,\xi) \ra [0,\infty]$ is \emph{upper Kud\textoverline{o}-continuous} if for every sequence $(\cA_n)$ in $\frak{S}(X,\xi)$,
we have
\[
\varlimsup_n F(\cA_n) \leq F(\cA^+),
\]
where $\cA^+$ denotes the minimal upper Kud\textoverline{o}-limit of $(\cA_n)$, and we say that $F$ is \emph{lower Kud\textoverline{o}-continuous} 
if for every sequence $(\cA_n)$ in $\frak{S}(X,\xi)$, we have
\[
 F(\cA^-) \leq \varliminf_n F(\cA_n),
\]
where $\cA^-$ denotes the maximal upper Kud\textoverline{o}-limit of $(\cA_n)$. We say that $F$ is \emph{Kud\textoverline{o}-continuous} if it is both upper and lower Kud\textoverline{o}-continuous.
\end{definition}

\begin{remark}
Since a sequence $(\cA_n)$ in $\frak{S}(X,\xi)$ converges to $\cA^+$ if and only if $\cA^+ \sim_\xi \cA^-$, we see that every Kud\textoverline{o}-continuous $F : \frak{S}(X,\xi) \ra [0,\infty)$ is continuous in the usual sense (with respect to the sequential topology induced from the strong operator topology on $B(L^1(X,\xi))$). 
\end{remark}

\subsection{Asymptotic second order domination and Kud\textoverline{o}-limits}

It turns out that we can link the notions of upper and lower Kud\textoverline{o}-limits of a sequence in $\frak{S}(X,\xi)$ to the more classical notion of second order stochastic domination between functions on $X$ as follows. We denote by $\cP$ the set of measurable functions $f : X \ra [0,\infty)$ 
such that $\int_X f \, d\xi = 1$. Given a measurable function $f \in \cP$, we 
set
\begin{equation}
\label{def_alphaf0}
\alpha_f(t) = \int_t^\infty \xi(\{ f \geq \tau\}) \, d\tau, \quad \textrm{for $t \geq 0$}.
\end{equation}
We recall the classical notion of second order stochastic domination: If $f_1, f_2 \in \cP$, then we say that $f_1$ is \emph{second order stochastically dominated by $f_2$} if 
\[
\alpha_{f_1}(t) \leq \alpha_{f_2}(t), \quad \textrm{for all $t \geq 0$}.
\]
There is also an asymptotic version of this relation: 

\begin{definition}[Upper and lower limits]
Let $(f_n)$ be a sequence in $\cP$ and let $f$ be an element in $\cP$. We say that $f$ is an \emph{upper limit} of $(f_n)$ if 
\[
\varlimsup_n \alpha_{f_n}(t) \leq \alpha_f(t), \quad \textrm{for all $t \geq 0$},
\]
\vspace{0,2cm}
and we say that $f$ is a lower limit of $(f_n)$ if 
\[
\alpha_f(t) \leq \varliminf_n \alpha_{f_n}(t), \quad \textrm{for all $t \geq 0$},
\]
\end{definition}

In Corollary \ref{cor_relateKudotoAsymp2nd} below, we relate the notions of upper and lower limits to Kud\textoverline{o}-limits of sub-$\sigma$-algebras.
More precisely, we show that if $(\cA_n)$ is a sequence in $\frak{S}(X,\xi)$ and $\cA^{+}$ and $\cA^{-}$ are elements in the space $\frak{S}(X,\xi)$, then 
\begin{enumerate}
\item[(i)] $\cA^{+}$ is an upper Kud\textoverline{o}-limit of $(\cA_n)$ if and only if for every $\rho \in \cP$, 
the conditional expectation $\bE_{\cA^{+}}(\rho)$ is an upper limit of the sequence $(\bE_{\cA_n}(\rho))$.
\vspace{0,2cm}
\item[(ii)] $\cA^{-}$ is an upper Kud\textoverline{o}-limit of $(\cA_n)$ if and only if for every $\rho \in \cP$, 
the conditional expectation $\bE_{\cA^{-}}(\rho)$ is a lower limit of the sequence $(\bE_{\cA_n}(\rho))$.
\end{enumerate}

\subsection{Entropy functions}
\label{subsec:entropyfcns}
The (standard) \emph{entropy} $\Ent_\xi$ 
is defined by
\[
\Ent_\xi(f) = \int_X f(x) \, \log f(x) \, d\xi(x), \quad \textrm{for $f \in \cP$}.
\]
Since the function $t \mapsto t \log t$ is convex on $(0,\infty)$ and equal to zero at $t = 1$, Jensen's inequality guarantees that $\Ent_\xi(f) \geq 0$
for all $f \in \cP$, and that $\Ent_\xi(f) = 0$ if and only if $f = 1$ $\xi$-almost everywhere. More generally, given a convex function 
$\Phi : [0,\infty) \ra \bR$ with $\Phi(1) = 1$, we define the \emph{$\Phi$-entropy} $\Ent_\xi^{\Phi} : \cP \ra [0,\infty]$ by
\[
\Ent_\xi^{\Phi}(f) = \int_X \Phi(f) \, d\xi, \quad \textrm{for $f \in \cP$}.
\]
In this paper, we shall assume that 
\vspace{0.1cm}
\begin{enumerate}
\item[(i)] $\Phi$ is twice continuously differentiable on $(0,\infty)$ and $\lim_{t \ra 0^+} t \, \Phi'(t) = 0$.
\vspace{0.1cm}
\item[(ii)] $\Phi(0) = 0$ and there exists a point $0 < t_o < 1$ such that $\Phi'$ is strictly negative on the interval $(0,t_o)$ and strictly positive on the
interval $(t_o,\infty)$.
\end{enumerate}
These properties are clearly satisfied by the function  
\[
\Phi(t) = 
\left\{
\begin{array}{cc}
0 & \textrm{for $t = 0$} \\[5pt]
t \log t & \textrm{for $t > 0$}
\end{array}
\right.,
\]
with $t_o = e^{-1}$. \\

We shall mostly work in the sub-space $\cP_\Phi \subset \cP$ defined as the set of all $f \in \cP$ such that
\[
\int_X |\Phi(f(x))| \, d\xi(x) < \infty \qand \lim_{t \ra \infty} \alpha_f(t) \, \Phi'(t) = 0,
\]
where $\alpha_f$ is given by \eqref{def_alphaf0}.

\section{Preliminaries on measured groups}
\label{sec:Borelprel}

\subsection{Borel $G$-spaces}
\label{subsec:BorelG}
Let $G$ be a locally compact, second countable and compactly generated group, and fix a \emph{left}-invariant Haar measure $m_G$ on $G$. 
Consider a Borel
action $G \times X \ra X$ on a standard Borel space $(X,\cB_X)$, equipped with a Borel probability measure
$\xi$, which is quasi-invariant with respect to the $G$-action, i.e. for every $g \in G$, the measures $g\xi$ and 
$\xi$ are equivalent.  In particular, this means that for every $g \in G$, the Radon-Nikodym derivative 
$\frac{dg\xi}{d\xi}$ exists $\xi$-almost everywhere. We shall refer to $(X,\xi)$ as a \emph{Borel $G$-space.} \\

Although it is not completely necessary for our arguments, it will be convenient to work with the following global version of $\frac{dg\xi}{d\xi}$:

\begin{lemma}{\cite[Theorem B.9]{Z}, \cite[Proposition 2.22]{FWMW05}}
\label{lemma_BorelG}
There exists a Borel measurable function $K : G \times X \ra [0,\infty)$ such that 
\begin{equation}
\label{RN1}
K(g_1 g_2,x) = K(g_1,x) K(g_2,g_1^{-1}x), \quad \textrm{for all $g_1,g_2 \in G$ and $x \in X$},
\end{equation}
and for every bounded measurable function $f$ on $X$, we have
\begin{equation}
\label{RN2}
\int_X f(gx) \, d\xi(x) = \int_X f(x) \, K(g,x) \, d\xi(x), \quad \textrm{for all $x \in X$},
\end{equation}
\end{lemma}
In what follows, we interchangeably write:
\begin{equation}
\label{notation}
K(g,x) = \rho^{\xi}_g(x) = \frac{dg\xi}{d\xi}(x), \quad \textrm{for $g \in G$ and $x \in X$}.
\end{equation}
We note that
\begin{equation}
\label{rho}
\rho^{\xi}_g \geq 0 \qand \int_X \rho^{\xi}_g\, d\xi = 1, \quad \textrm{for all $g \in G$}.
\end{equation}
We recall that a \emph{$G$-factor} of $(X,\xi)$ is another Borel $G$-space $(Z,\theta)$, together with a $\xi$-conull $G$-invariant 
subset $X'' \subset X$ and a $G$-equivariant measurable map $\pi : X'' \ra Z$ such that $\pi_*(\xi \mid_{X''}) = \nu$. We shall sometimes 
abuse notation and write $\pi : (X,\xi) \ra (Z,\theta)$ for the $G$-factor, suppressing the dependence on $X''$, and we note that the pull-back 
$\pi^{-1}(\cB_Z)$ of the Borel $\sigma$-algebra $\cB_Z$ is a $G$-invariant sub-$\sigma$-algebra of $\cB_X$ (restricted to the conull subset 
$X''$). The following lemma is standard; we only give a proof below for completeness.

\begin{lemma}
\label{lemma_RNfactor}
Let $\pi : (X,\xi) \ra (Z,\theta)$ be a $G$-factor between two Borel $G$-space. Then, for all $g \in G$,
we have
\[
\rho_g^{\theta} \circ \pi = \bE_{\pi^{-1}(\cB_Z)}(\rho_g^{\xi}) \quad \textrm{in $L^1(X,\xi)$}.
\]
\end{lemma}

Suppose now that $((Z_n,\theta_n))$ is a sequence of $G$-factors of $(X,\xi)$, together with $G$-equivariant Borel maps 
$\pi_n : (X,\xi) \ra (Z_n,\theta_n)$, as above, for every $n$. Then $\cA_n := \pi_{Z_n}^{-1}(\cB_{Z_n})$ is a
$G$-invariant sub-sigma algebra of $\cB_X$ (modulo $\xi$-null sets) for every $n$, and we denote by $\cA^{-}$ and $\cA^{+}$ the 
\emph{maximal} lower and \emph{minimal} upper Kud\textoverline{o}-limits of the sequence $(\cA_n)$ respectively. These are uniquely 
determined up to $\xi$-null sets. The following lemma is proved below.

\begin{lemma}
\label{Kudoinvariant}
The sub-$\sigma$-algebras $\cA^{-}$ and $\cA^{+}$ are $G$-invariant (modulo $\xi$-null sets).
\end{lemma}

By Mackey's Point Realization Theorem (see e.g. \cite{Mack}), we can now conclude that there exist Borel $G$-spaces $(Z^{-},\theta^{-})$ and $(Z^{+},\theta^{+})$, together with
$G$-equivariant measurable maps 
\[
\pi_{-} : (X,\xi) \ra (Z^{-},\theta^{-}) \qand \pi_{+} : (X,\xi) \ra (Z^{+},\theta^{+})
\]
such that $\pi^{-1}_{-}(\cB_{Z^{-}}) = \cA^{-}$ and $\pi_{+}^{-1}(\cB_{Z^{+}}) = \cA^{+}$  modulo $\xi$-null sets. Moreover, these spaces and maps 
are uniquely determined up to $G$-equivariant isomorphisms, that is to say, if $(W^{-},\lambda^{-})$ and $(W^{+},\lambda^{+})$ are two other
Borel $G$-spaces, endowed with $G$-equivariant measurable maps 
\[
\tau_{-} : (X,\xi) \ra (W^{-},\lambda^{-}) \qand \tau_{+} : (X,\xi) \ra (W^{+},\lambda^{+})
\]
such that $\tau^{-1}_{-}(\cB_{W^{-}}) = \cA^{-}$ and $\tau_{+}^{-1}(\cB_{W^{+}}) = \cA^{+}$  modulo $\xi$-null sets, then there are $G$-equivariant
measurable isomorphism $\kappa_{\pm} : (Z^{\pm},\theta^{\pm}) \ra (W^{\pm},\lambda^{\pm})$ such that $\tau_{\pm} = \kappa_{\pm} \circ \pi_{\pm}$.
We shall refer to the $G$-factors $(Z^{-},\theta^{-})$ and $(Z^{+},\theta^{+})$ as the \emph{lower} and \emph{upper} Kud\textoverline{o}-limits of 
the sequence $((Z_n,\theta_n))$ respectively. 

\subsection{Measured groups and stationary measures}
\label{subsec:measuredgrps}
Suppose that $\mu$ is a Borel probability measure on $G$ which is absolutely continuous with respect to $m_G$. We write $d\mu = u \, dm_G$ and shall assume throughout the rest of this section that $u$ is a continuous function on $G$ with relatively compact support such that
\begin{equation}\label{generating2}
G=\bigcup_{k \geq 1} \big\{ u^{*k} > 0 \big\}.
\end{equation}
We shall refer to the pair $(G,\mu)$ as a \emph{measured group}.
In particular, \eqref{generating2} implies that the Borel probability measure $\eta_\mu$ on $G$ defined by
\begin{equation}
\label{def_etamu}
\eta_\mu = \sum_{k \geq 1} 2^{-k} \, \mu^{*k},
\end{equation}
is absolutely continuous with respect to $m_G$ with an everywhere positive density. \\

Let $(X,\xi)$ be a Borel $G$-space. We shall throughout the rest of this section assume that $\xi$ is \emph{$\mu$-stationary}, 
i.e. 
\[
\mu * \xi = \int_G g \xi \, d\mu(g) = \xi,
\] 
in which case we refer to $(X,\xi)$ as a \emph{Borel $(G,\mu)$-space}.

The following lemma will be proved below.
\begin{lemma}[Harnack's inequality]
\label{lemma_RNbounded}
For $(G,\mu)$ as above   every $\mu$-stationary space  $(X,\xi)$
 has bounded Radon-Nikodym derivatives, i.e.  for every $g \in G$ there exists
$\lambda_g \geq 1$ such that
\[
\lambda_g^{-1} \leq \rho_g^\xi(x) \leq \lambda_g,
\]
for $\xi$-almost every $x \in X$. Furthermore, the map $g \mapsto \lambda_g$ is locally bounded on $G$, i.e. bounded on compact subsets of $G$.
\end{lemma}

\subsection{Boundary theory of measured groups}
\label{subsec:bndtheory}
We say that $\psi \in L^\infty(G, \eta_\mu)$ is \emph{$\mu$-harmonic} if 
\[
(\psi * \mu)(g) = \int_G \psi(gh) \, d\mu(h) = \psi(g), \quad \textrm{for all $g \in G$}.
\]
Since we assume that $\mu$ is absolutely continuous with respect to the Haar measure on $G$, every $\mu$-harmonic
function must be continuous (in fact, right uniformly continuous, see \cite[Lemma 1.2]{Bab}). We denote by $\cH^\infty(G,\mu)$ the space of $\mu$-harmonic functions on $G$. Given
a Borel $(G,\mu)$-space $(X,\xi)$, we consider the \emph{Poisson transform} 
\begin{equation}
\label{def_Poissonmap}
P_\xi : L^\infty(X,\xi) \ra L^\infty(G, \eta_\mu), \quad f \mapsto \widehat{\psi}_f(g) = \int_X f(gx) \, d\xi(x), \enskip \textrm{for $g \in G$}.
\end{equation}
It readily follows from $\mu$-stationarity of $\xi$ that the image of $P_\xi$ is always contained in $\cH^\infty(G,\mu)$. A fundamental
observation of Furstenberg (see e.g. \cite[Theorem 2.13]{Fur}) is that for every measured group, there is always a Borel $(G,\mu)$-space $(B,\nu)$ such that the
Poisson transform is an isometric isomorphism between the Banach spaces $L^\infty(B,\nu)$ and $\cH^\infty(G,\mu)$, where the latter 
inherits the norm from the ambient $L^\infty(G,\eta_\mu)$-space. This Borel $(G,\mu)$-space is further unique up to measurable and 
$G$-equivariant isomorphisms, and it (or any of its isomorphic versions) is referred to as the \emph{Poisson boundary} of $(G,\mu)$. 
A $G$-factor of the Poisson boundary of $(G,\mu)$ is called a \emph{$\mu$-boundary}.  \\

The following lemma is proved below.
\begin{lemma}
\label{lemma_poissonmap}
Let $(B,\nu)$ be the Poisson boundary of $(G,\mu)$. Then the map $\psi \mapsto f_\psi$ from $L^\infty(G, \eta_\mu)$ to $L^1(B, \nu)$
given by
\[
f_\psi = \int_G \psi(g) \, \rho_g^{\nu}  \, d\eta_\mu(g), \quad \textrm{for $\psi \in L^\infty(G,\eta_\mu)$}
\]
has a norm-dense image $V \subset L^1(B, \nu)$.
\end{lemma}

\subsection{The Furstenberg entropy of a stationary action}

The \emph{Furstenberg $\mu$-entropy} $h_{(G,\mu)}(X,\xi)$ of the Borel $G$-space $(X,\xi)$ is defined by
\[
h_{(G,\mu)}(X,\xi) = \int_G \Big( \int_X -\log \frac{dg^{-1}\xi}{d\xi}(x) \, d\xi(x) \Big) \, d\mu(g).
\]
It is well-known (see e.g. \cite[Subsection 2.7]{Fur}) that $h_{(G,\mu)}(X,\xi) = 0$ if and only if $\xi$ is $G$-invariant. Furthermore, if $(B,\nu)$ denotes the
Poisson boundary of $(G,\mu)$, then
\[
h_{(G,\mu)}(B,\nu) \geq h_{(G,\mu)}(Z,\theta), \quad \textrm{for every $\mu$-boundary $(Z,\theta)$},
\]
with equality if and only if $(Z,\theta)$ is isomorphic to $(B,\nu)$ via a measurable $G$-equivariant isomorphism. We further prove below:
\begin{lemma}
\label{lemma_entropy}
Let $(X,\xi)$ be a Borel $(G,\mu)$-space and set $\gamma = \sum_{k=1}^\infty \frac{k}{2^{k}}$. 
Then, 
\vspace{0.2cm}
\begin{enumerate}
\item[(i)] with $\eta_\mu$ as defined  in \eqref{def_etamu}, 
\[
h_{(G,\mu)}(X,\xi) = \gamma^{-1} \, \int_G \Ent_{\xi}(\rho^{\xi}_g) \, d\eta_\mu(g).
\]
In particular, if $h_{(G,\mu)}(X,\xi)$ is finite, then 
so is $ \int_G \Ent_{\xi}(\rho^{\xi}_g) \, d\eta_\mu(g)$.
\vspace{0.2cm}
\item[(ii)] if $\pi : (X,\xi) \ra (Z,\theta)$ is a $G$-factor, we have
\[
h_{(G,\mu)}(Z,\theta) = \gamma^{-1} \, \int_G \Ent_{\xi}\big(\bE_{\pi_Z^{-1}(\cB_Z)}(\rho^{\xi}_g)\big) \, d\eta_\mu(g)
\]
\end{enumerate}
\end{lemma}

These alternative formulas relate the Furstenberg entropy to the notion of averaged entropy of a sub-$\sigma$-algebra, as discussed in connection with Theorem \ref{thm_kernelconv}. 

\subsection{Proof of Lemma \ref{lemma_RNfactor}}

Let $(X,\xi)$ be a Borel $G$-space and suppose that $\pi : (X,\xi) \ra (Z,\theta)$ is a $G$-factor. We shall prove that for all $g \in G$,
\[
\rho_g^{\theta} \circ \pi = \bE_{\pi^{-1}(\cB_Z)}(\rho_g^{\xi}), \quad \textrm{$\xi$-almost everywhere},
\]
or equivalently, that for every $g \in G$ and $f \in L^\infty(Z,\theta)$, we have 
\[
\int_X (f \circ \pi)  \, (\rho_g^\theta \circ \pi) \, d\xi =  \int_X (f \circ \pi)  \,  \bE_{\pi_Z^{-1}(\cB_Z)}(\rho_g^\xi) \, d\xi.
\]
To prove this, pick $g \in G$ and $f \in L^\infty(Z, \theta)$, and note that
\begin{eqnarray*}
\int_X (f \circ \pi)  \, (\rho_g^\theta \circ \pi) \, d\xi,
&=&
\int_Z f \, \rho_g^\theta \, d\theta 
=
\int_Z f(gz) \, d\theta(z) \\
&=&
\int_X (f \circ \pi)(gx) \, d\xi(x) 
=
\int_X (f \circ \pi) \, \rho_g^{\xi} \, d\xi \\[1pt]
&=&
\int_X (f \circ \pi)  \, \bE_{\pi_Z^{-1}(\cB_Z)}(\rho^\xi_g) \, d\xi,
\end{eqnarray*}
which finishes the proof.

\subsection{Proof of Lemma \ref{Kudoinvariant}}

Let $(X,\xi)$ be a Borel $G$-space, and suppose that $(\cA_n)$ is a sequence of \emph{$G$-invariant} 
sub-$\sigma$-algebras of $\cB_X$. We denote by $\cA^{+}$ and $\cA^{-}$ the minimal upper and maximal lower Kud\textoverline{o}-limits of 
the sequence $(\cA_n)$ respectively. We shall show that both $\cA^{+}$ and $\cA^{-}$ are $G$-invariant. \\

Let us begin with the proof that $\cA^{-}$ is $G$-invariant. By \cite[Theorem 3.2]{Kudo}, we have
\[
\cA^{-} = \big\{ A \in \cB_X \mid \textrm{there exists $A_n \in \cA_n$ such that $\xi(A_n \Delta A) \ra 0$} \big\}.
\]
Pick $A \in \cA^{-}$ and a sequence $A_n \in \cA_n$ such that $\xi(A_n \Delta A) \ra 0$. Note that this
implies that the sequence $(\chi_{A_n \Delta A})$ converges to zero in the weak* topology on $L^\infty(X,\xi)$, and thus
\[
\xi(gA \Delta gA_n) = \xi(g(A \Delta A_n)) = \int_X \chi_{A \Delta A_n} \, \rho_g^{\xi} \, d\xi \ra 0, \quad \textrm{for all $g \in G$}.
\]
Since $g A_n \in \cA_n$, this shows that $gA \in \cA^{-}$ (modulo $\xi$-null sets), whence $\cA^{-}$ is $G$-invariant. \\

The proof that $\cA^{+}$ is $G$-invariant is a bit more involved, and requires the following lemma.

\begin{lemma} 
\label{transform}
For every sub-$\sigma$-algebra $\cA \subset \cB_X$, we have
\[
\bE_{g\cA}(f) = \Big( \frac{\bE_{\cA}\big(\rho_{g^{-1}}^\xi \, (f \circ g)\big) }{\bE_\cA(\rho_{g^{-1}}^{\xi})} \Big) \circ g^{-1}, 
\]
for every $g \in G$ and $f \in L^{1}(X,\xi)$. 
In particular $\Vert \bE_{g\cA}(f) \Vert_{L^1(\xi)} = \Vert \bE_{\cA}\big(\rho_{g^{-1}}^\xi \, (f \circ g)\big) \Vert_{L^1(\xi)} $ for every $g \in G$ and $f \in L^{1}(X,\xi)$.
\end{lemma}

\begin{proof}
Let $\cA \subset \cB_X$ be a sub-$\sigma$-algebra. Fix $g \in G$ and $f \in L^1(X,\xi)$, and note that for every $A \in \cA$,
\begin{eqnarray*}
\int_{gA} \bE_{g\cA}(f) \, d\xi
&=&
\int_{gA} f \, d\xi = \int_A \rho_{g^{-1}}^\xi \, (f \circ g) \, d\xi \\
&=&
\int_A \bE_{\cA}\big(\rho_{g^{-1}}^\xi \, (f \circ g)) \, d\xi = 
\int_A \frac{\bE_{\cA}\big(\rho_{g^{-1}}^\xi \, (f \circ g))}{\bE_{\cA}(\rho_{g^{-1}}\xi)} \, \bE_{\cA}(\rho_{g^{-1}}^\xi) \, d\xi \\
&=&
\int_A \frac{\bE_{\cA}\big(\rho_{g^{-1}}^\xi \, (f \circ g))}{\bE_{\cA}(\rho_{g^{-1}}^\xi)} \, \rho_{g^{-1}}^\xi \, d\xi
=
\int_{gA} \Big(\frac{\bE_{\cA}\big(\rho_{g^{-1}}^\xi \, (f \circ g))}{\bE_{\cA}(\rho_{g^{-1}}^\xi)} \Big) \circ g^{-1} \, d\xi.
\end{eqnarray*}
Since the integrand in the last integral is $g\cA$-measurable, and $A \in \cA$ is arbitrary, the first identity is proven.
\\ Now let us consider the above identity when applying the norm. Given $f\in L^{1}(X,\xi)$ we see that 
\[
\int_X \vert \bE_{g\cA}(f) \vert \, d\xi = 
\int_X   \frac{\vert \bE_{\cA}\big(\rho_{g^{-1}}^\xi \, (f \circ g)\big) \vert }{ \bE_\cA(\rho_{g^{-1}}^{\xi}) } \, \rho^{\xi}_{g^{-1}} \, d\xi =
\int_X  \vert \bE_{\cA}\big(\rho_{g^{-1}}^\xi \, (f \circ g)\big) \vert   \, d\xi
\] as claimed.
\end{proof}

To show that $\cA^{+}$ is $G$-invariant, it suffices to prove that for every $g \in G$, the sub-$\sigma$-algebra $g\cA^{+}$ is again 
an upper limit of $(\cA_n)$. Indeed, assume that we know this, then for every $g \in G$, the intersection $\cA^{+} \cap g \cA^{+}$ is again an upper limit of $(\cA_n)$ by \cite[Lemma 3.2 (i)]{Kudo},
whence must coincide with the Kud\textoverline{o} upper limit $\cA^{+}$ by minimality of the latter. 
\\So given $g\in G$ let us show  that $g\cA^{+}$ is an upper limit for $(\cA_n)$. Let $f\in L^1(X,\xi)$ be arbitrary. Applying Lemma~\ref{transform} two times and using $G$-invariance of $(\cA_n)$ we see that \[\Vert \bE_{g\cA^{+}}(f)\Vert_{L^1(\xi)}= \Vert  \bE_{\cA^+}\big(\rho_{g^{-1}}^\xi \, (f \circ g)\big)  \Vert_{L^1(\xi)} \geq 
\varlimsup_{n}\Vert  \bE_{\cA_n}\big(\rho_{g^{-1}}^\xi \, (f \circ g)\big)  \Vert_{L^1(\xi)} 
\]
\[=\varlimsup_{n}\Vert  \bE_{g\cA_n}\big(f\big)  \Vert_{L^1(\xi)}
=\varlimsup_{n}\Vert  \bE_{\cA_n}\big(f\big)  \Vert_{L^1(\xi)},
\]
thus $g\cA^+$ is an upper limit of $(\cA_n)$.

\subsection{Proof of Lemma \ref{lemma_RNbounded}}\label{sec: rn}
We assume that $d\mu = u \, dm_G$, where $u$ is a continuous function on $G$ with  $\{u>0\}\subseteq Q$ for a   compact set $Q$ and such that\[G = \bigcup_{k \geq 1} \big\{ u^{*k} > 0 \big\}.\]
The following lemma is the main ingredient in the proof of Lemma~\ref{lemma_RNbounded}.

\begin{lemma}
\label{Harnack}
For every $g \in G$, there exists a strictly positive constant $C_g$ such that for every measurable non-negative function 
$\psi$ on $G$ with the property that 
\[
(\psi * \mu)(g) = \psi(g) \qand (\psi * \mu)(e) = \psi(e) = 1,
\]
we have $\psi(g) \leq C_g$.
\end{lemma}

\begin{proof}[Proof of Lemma \ref{lemma_RNbounded} assuming Lemma \ref{Harnack}]
Let us fix $g \in G$ throughout the proof. Since $\xi$ is $\mu$-stationary, we see that for every $f \in L^\infty(X,\xi)$,
\begin{eqnarray*}
\int_X f(x) \, \rho_g^{\xi}(x) \, d\xi(x) 
&=& 
\int_X f(gx) \, d\xi(x) = \int_X f(ghx) \, d\mu(h) \, d\xi(x) \\
&=&
\int_X \int_G f(ghx) \, d\mu(h) \, d\xi(x) \\
&=& \int_X f(x) \, \Big( \int_G  \rho^\xi_{gh}(x) \, d\mu(h) \Big) \, d\xi(x),
\end{eqnarray*}
whence there exists a $\xi$-conull subset $X_g \subset X$ such that 
\[
\int_G \rho^\xi_{gh}(x) \, d\mu(h) = \rho^\xi_g(x), \quad \textrm{for all $x \in X_g$}.
\]
Since $\rho_e^\xi(x) = 1$ for all $x$, we conclude that for every $x \in X_g \cap X_e$, the function 
\[
\psi_x(s) = \rho_s^\xi(x),\quad \textrm{for $s \in G$},
\]
satisfies the assumptions of Lemma \ref{Harnack}, whence $\rho_g^\xi(x) \leq C_g$ for all $x \in X_g \cap X_e$. Furthermore, 
we note that the
identities in \eqref{RN1} now  imply that 
\[
\rho_g(x) = \frac{1}{\rho_{g^{-1}}(gx)} \geq \frac{1}{C_{g^{-1}}}, \quad \textrm{for all $x \in X_{g^{-1}} \cap X_e$}.
\]
If we set $\lambda_g = \max(C_g,C_{g^{-1}})$, then 
\[
\lambda_g^{-1} \leq \rho_g^{\xi}(x) \leq \lambda_g, \quad \textrm{for all $x \in X_g \cap X_{g^{-1}} \cap X_e$}.
\]
Since $X_g, X_{g^{-1}}$ and $X_e$ are all $\xi$-conull sets, we have finished the proof. 
\end{proof}

\begin{proof}[Proof of Lemma \ref{Harnack}.]
We shall need some notation. For $k \geq 1$, we set
\[
S_k = \big\{ g \in G \, \mid \, u^{*k}(g) > 0 \big\}.
\]
Since $u^{*k}$ is a continuous function on $G$, we note that $S_k$ is open and for every $s \in G$, the set 
\[
V_k(s) = \Big\{ t \in G \, \mid \, u^{*k}(st) > \frac{1}{2} u^{*k}(s) \Big\}
\]
is an open identity neighbourhood in $G$. If $u^{*k}(s) > 0$, then $V_k(s)$ is non-empty. Since $G = \bigcup_k S_k$, 
we have
\[
G = \bigcup_{s \in G} \bigcup_{k \geq 1} s \, V_k(s).
\]
Since the support of $u$ is contained in a compact set $Q$, there exists a finite set $F_g \subset G \times \bN$ such that
\begin{equation}
\label{Q}
g\, \{u>0\}\subseteq gQ \subseteq \bigcup_{(s,k) \in F_g} s\, V_k(s),
\end{equation}
which has the property that $u^{*k}(s) > 0$ for all $(s,k) \in F_g$. \\

Let us now fix $g \in G$ and a measurable non-negative function $\psi$ on $G$ such that 
\[
(\psi * \mu)(g) = \psi(g) \qand (\psi * \mu)(e) = \psi(e) = 1.
\]
Since $(\psi * \mu)(g) = \psi(g)$ and $m_G$ is left-invariant, 
we have
\begin{eqnarray*}
\psi(g) 
&=&
\int_G \psi(gh) \, u(h) \, dm_G(h) \leq \|u\|_\infty \int_{\{u>0\}} \psi(gh) \, dm_G(h) \\[2pt]
&=& \|u\|_\infty  \, \int_{g\{u>0\}} \psi(h) \, dm_G(h) \leq
\|u\|_\infty  \, \sum_{(s,k) \in F_g} \int_{s \, V_k(s)} \psi(h) \, dm_G(h) 
\\[2pt]
&=&
\|u\|_\infty  \, \sum_{(s,k) \in F_g}  \, \int_{V_k(s)} \psi(sh) \, dm_G(h).
\end{eqnarray*}
Using the bound 
\[
u^{*k}(sh) \geq \frac{1}{2} u^{*k}(s) > 0, \quad \textrm{for all $h \in V_{k}(s)$},
\] 
and the non-negativity of $\psi$, we now get
\begin{eqnarray*}
\psi(g) &\leq & \|u\|_\infty  \, \sum_{(s,k) \in F_g}  \, \int_{V_k(s)} \psi(sh) \, \frac{u^{*k}(sh)}{u^{*k}(sh)}\, dm_G(h) \\[2pt]
&\leq &
2 \, \|u\|_\infty \, \, \sum_{(s,k) \in F_g} \frac{1}{u^{*k}(s)} \, \int_{V_k(s)} \psi(sh) \, u^{*k}(sh)\, dm_G(h) \\[2pt]
&\leq &
2 \, \|u\|_\infty \, \, \sum_{(s,k) \in F_g} \frac{1}{u^{*k}(s)} \, \int_{G} \psi(sh) \, u^{*k}(sh)\, dm_G(h) \\[2pt]
&= &
2 \, \|u\|_\infty \, \, \sum_{(s,k) \in F_g} \frac{1}{u^{*k}(s)} \, \int_{G} \psi(h) \, u^{*k}(h)\, dm_G(h),
\end{eqnarray*}
where we in the last step used the left-invariance of $m_G$. Since 
\[
\int_G \psi(h) u^{*k}(h) \, dm_G(h) = \psi(e) = 1,
\]
we can now conclude that
\[
\psi(g) \leq C_g := 2 \, \|u\|_\infty \, \, \sum_{(s,k) \in F_g} \frac{1}{u^{*k}(s)}.
\]
Moreover, $g\mapsto C_g$ is locally bounded since Eq.~\eqref{Q} can be adapted to  the case when we run over a collection of elements $g$  taken from  a compact set.
\end{proof}

\subsection{Proof of Lemma \ref{lemma_poissonmap}}

Let $(G,\mu)$ be a measured group and let $(B,\nu)$ denote the Poisson boundary of $(G,\mu)$. We shall prove that the set $V$ of all functions of 
the form
\[
f_\psi = \int_G \psi(g) \, \rho_g^{\nu}  \, d\eta_\mu(g), 
\]
as $\psi$ ranges over $L^\infty(G, \eta_\mu)$ is norm-dense in $L^1(B, \nu)$. Assume that this is not the case. Then, by Hahn-Banach's Theorem, there exists
a non-zero $\phi \in  (L^1(B,\nu))^* \cong L^\infty(B,\nu)$ such that 
\begin{eqnarray*}
\int_B f_\psi \, \phi \, d\nu 
&=& 
\int_G \int_B \psi(g) \rho_g^{\nu}(b) \, \phi(b) \, d\nu(b) \, d\eta_\mu(g) \\
&=& 
\int_G \int_B \psi(g) \, \phi(gb) \, d\nu(b) \, d\eta_\mu(g) \\
&=&
\int_G \psi(g) (P_\nu \phi)(g) \, d\eta_\mu(g) = 0,
\end{eqnarray*}
for all $\psi \in L^\infty(G,\eta_\mu)$, where $P_\nu$ denotes the Poisson transform of $(B,\nu)$, introduced in \eqref{def_Poissonmap}. This readily shows that $P_\nu \phi = 0$ in $L^\infty(G,\eta_\mu)$. Since $(B,\nu)$ is the Poisson boundary of $(G,\mu)$, the map $P_\nu$ is an isometric isomorphism 
from $L^\infty(B,\nu)$ into $L^\infty(G, \eta_\mu)$, whence $\phi = 0$, which is a contradiction. This finishes the proof. 

\subsection{Proof of Lemma \ref{lemma_entropy}}

By Lemma \ref{lemma_BorelG}, we have 
\begin{equation}
\label{g1g2}
\frac{d(g_1 g_2)^{-1}\xi}{d\xi}(x) = \frac{dg_1^{-1}\xi}{d\xi}(g_2 x) \, \frac{dg_2^{-1}\xi}{d\xi}(x), 
\end{equation}
for all $g_1, g_2 \in G$ and $\xi$-almost every $x \in X$. Hence, if $\mu_1$ and $\mu_2$ are two Borel probability measures
on $G$ such that $\mu_1 * \xi = \mu_2 * \xi = \xi$, then
\begin{eqnarray*}
\int_G \int_X -\log \frac{dg^{-1}\xi}{d\xi}(x) \, d\xi(x) \, d(\mu_1 *\mu_2)(g) &=& 
\int_G \int_X -\log \frac{dg_1^{-1}\xi}{d\xi}(x) \, d\xi(x) \, d\mu_1(g_1) \\
&+& 
\int_G \int_X -\log \frac{dg_2^{-1}\xi}{d\xi}(x) \, d\xi(x) \, d\mu_2(g_2).
\end{eqnarray*}
In particular,
\[
\int_G \int_X -\log \frac{dg^{-1}\xi}{d\xi}(x) \, d\xi(x) \, d\mu^{*k}(g) = k \int_G \int_X -\log \frac{dg^{-1}\xi}{d\xi}(x) \, d\xi(x) \, d\mu(g),
\]
for all $k \geq 1$, whence, with the Borel probability measure $\eta_\mu$ on $G$ defined as
\[
\eta_\mu = \sum_{k=1}^{\infty} \frac{1}{2^{k}} \, \mu^{*k},
\]
we have
\[
\int_G \int_X -\log \frac{dg^{-1}\xi}{d\xi}(x) \, d\xi(x) \, d\eta_\mu(g) = \gamma \, \int_G \int_X -\log \frac{dg^{-1}\xi}{d\xi}(x) \, d\xi(x) \, d\mu(g),
\]
where
\[
\gamma = \sum_{k=1}^\infty \frac{k}{2^{k}}.
\]
It also follows from \eqref{g1g2} that for every $g \in G$,
\[
\frac{dg^{-1}\xi}{d\xi}(x) = \frac{1}{\frac{dg\xi}{d\xi}(gx)}, \quad \textrm{$\xi$-almost everywhere},
\]
whence 
\begin{eqnarray*}
\int_G \int_X -\log \frac{dg^{-1}\xi}{d\xi}(x) \, d\xi(x) \, d\mu(g) 
&=& 
\gamma^{-1} \, \int_G \int_X \log \frac{dg\xi}{d\xi}(gx) \, d\xi(x) \, d\eta_\mu(g) \\
&=& 
\gamma^{-1} \, \int_G \int_X \frac{dg\xi}{d\xi}(x) \, \log \frac{dg\xi}{d\xi}(x) \, d\xi(x) \, d\eta_\mu(g) \\
&=&
\gamma^{-1} \, \int_G \Ent_\xi(\rho_g^{\xi}) \, d\eta_\mu(g),
\end{eqnarray*}
which finishes the proof of (i). Now (ii) follows from Lemma \ref{lemma_RNfactor}.

\section{Proof of Theorem  \ref{thm_upper/lower-limits}}
\label{sec:proofUppLow}
Throughout this section, we adopt the notation and assumptions in Subsection \ref{subsec:entropyfcns}.
The following proposition is the main ingredient in the proof of Theorem \ref{thm_upper/lower-limits}. 
\vspace{0.1cm}
\begin{proposition} 
\label{Prop_Phi-inequality}
For every $0 < \delta < t_o$, we have
\[
\sup_{f \in \cP_\Phi} \Big| \Ent_\xi^{\Phi}(f) - \int_\delta^\infty \alpha_f(t) \, \Phi''(t) \, dt - \Phi'(\delta) \Big| \leq -2 \max\big( \Phi(\delta), \delta \, \Phi'(\delta) \big).
\]
\end{proposition}
\vspace{0.1cm}

\begin{remark}
We stress that $\Phi'(\delta)$ is typically unbounded as $\delta \ra 0^+$. For instance, if $\Phi(t) = t \log t$, then $\Phi'(\delta) = 1 + \log \delta$,
which tends to $-\infty$ as $\delta \ra 0^+$. This is why we must always keep track of the term $\Phi'(\delta)$ throughout our estimates.
\end{remark}

\begin{proof}[Proof of Theorem \ref{thm_upper/lower-limits} assuming Proposition \ref{Prop_Phi-inequality}]
Fix $\eps > 0$. Since $\Phi$ is a continuous function with $\Phi(0) = 0$, and $\lim_{t \ra 0^+} t \Phi'(t) = 0$, we can choose $0 < \delta < t_o$ such that
\[
-2 \max\big( \Phi(\delta), \delta \, \Phi'(\delta) \big) < \eps/2.
\]
Let $(f_n)$ be a sequence in $\cP_\Phi$ and let $f$ be an element in $\cP_\Phi$. By Proposition \ref{Prop_Phi-inequality}, we have
\begin{equation}
\label{sandwich1}
\int_\delta^\infty \alpha_{f_n}(t) \, \Phi''(t) \, dt + \Phi'(\delta) - \eps/2 \leq \Ent_\xi^\Phi(f_n) \leq \int_\delta^\infty \alpha_{f_n}(t) \, \Phi''(t) \, dt 
+ \Phi'(\delta) + \eps/2,
\end{equation}
for all $n$, as well as,
\begin{equation}
\label{sandwich2}
\int_\delta^\infty \alpha_{f}(t) \, \Phi''(t) \, dt + \Phi'(\delta) - \eps/2 \leq \Ent_\xi^\Phi(f) \leq \int_\delta^\infty \alpha_{f}(t) \, \Phi''(t) \, dt 
+ \Phi'(\delta) + \eps/2.
\end{equation}
Since $\Phi$ is a convex and twice differentiable function on $(0,\infty)$, we have $\Phi'' > 0$ on this interval, whence 
$dw_o(t) = \Phi''(t) \, dt$ is a non-negative measure on $[\delta,\infty)$. \\

Let us first prove (i). We assume that $f$ is a lower limit of the sequence $(f_n)$ so that
\[
\alpha_f(t) \leq \varliminf_n \alpha_{f_n}(t), \quad \textrm{for all $t \geq 0$}.
\]
By Fatou's Lemma, applied to the first inequality in \eqref{sandwich1}, we conclude that
\begin{eqnarray*}
\int_\delta^\infty \alpha_{f}(t) \, \Phi''(t) \, dt + \Phi'(\delta) - \eps/2 
&\leq &
\varliminf_n \int_\delta^\infty \alpha_{f_n}(t) \, \Phi''(t) \, dt + \Phi'(\delta) - \eps/2 \\
& \leq & \varliminf_n \Ent_\xi^\Phi(f_n).
\end{eqnarray*}
Now the second inequality in \eqref{sandwich2} implies that
\[
\Ent_\xi^\Phi(f) - \eps \leq \int_\delta^\infty \alpha_{f}(t) \, \Phi''(t) \, dt + \Phi'(\delta) - \eps/2 \leq  \varliminf_n \Ent_\xi^\Phi(f_n).
\]
Since $\eps > 0$ is arbitrary, we have proved (i). \\

To prove (ii), let us assume that $f$ is an upper limit of $(f_n)$, so that 
\[
\varlimsup_n \alpha_{f_n}(t) \leq \alpha_f(t) \quad \textrm{for all $t \geq 0$}.
\]
We shall in addition assume that for some $\beta > 0$,
\[
\int_1^\infty t^{-\beta} \, \Phi''(t) \, dt < \infty \qand \sup_n \int_X f_n^{1+\beta} \, d\xi < \infty.
\]
The first condition says that the measure $dw_\beta(t) = t^{-\beta} \Phi''(t) \, dt$ is a finite non-negative measure on the interval $[\delta,\infty)$ 
(since $\Phi''$ is bounded on $[\delta,1)$),
while the second condition implies that the sequence $(\varphi_n)$ of non-negative functions on $[\delta,\infty)$ defined by $\varphi_n(t) := t^{\beta} \alpha_{f_n}(t)$ satisfies
\[
\sup_n \|\varphi_n\|_\infty \leq C_\beta := \frac{1}{\beta} \sup_n \int_X f_n^{1+\beta} \, d\xi < \infty.
\]
To see this, note that Markov's inequality implies that
\[
\alpha_{f_n}(t) \leq \frac{1}{\beta t^\beta} \int_X f_n^{1+\beta} \, d\xi, \quad \textrm{for all $n$}.
\]
If we now apply Fatou's Lemma to the non-negative sequence $(C_\beta-\varphi_n)$ and the measure $dw_\beta$, we conclude that
\begin{equation}
\label{limsup}
\varlimsup_n \int_\delta^\infty \varphi_n \, dw_\beta(t) \leq \int_0^\infty \varlimsup_n \varphi_n(t) \, dw_\beta(t) \leq \int_\delta^\infty \alpha_f(t) \, \Phi''(t) \, dt,
\end{equation}
since $f$ is an upper limit of $(f_n)$. We now observe that the second inequality in \eqref{sandwich1}, together with \eqref{limsup} implies that
\begin{eqnarray*}
\varlimsup_n \Ent_\xi^\Phi(f_n) 
&\leq & 
\varlimsup_n \int_\delta^\infty \alpha_{f_n}(t) \, \Phi''(t) \, dt + \Phi'(\delta) + \eps/2 \\
&=&
\varlimsup_n \int_\delta^\infty \varphi_n \, dw_\beta(t) + \Phi'(\delta) + \eps/2 \\
&\leq &
\int_\delta^\infty \alpha_f(t) \, \Phi''(t) \, dt + \Phi'(\delta) + \eps/2.
\end{eqnarray*}
The first inequality in \eqref{sandwich2} now shows that
\[
\varlimsup_n \Ent_\xi^\Phi(f_n) \leq \Ent^\Phi_\xi(f) + \eps,
\]
and since $\eps > 0$ is arbitrary, we have proved (ii).
\end{proof}

\subsection{Proof of Proposition \ref{Prop_Phi-inequality}}

\begin{lemma}
\label{Lemma_EntAlternative}
For every $f \in \cP_\Phi$,
\[
\Ent_\xi^\Phi(f) = \int_0^\infty \xi(\{f \geq u\}) \, \Phi'(u) \, du.
\]
\end{lemma}

\begin{remark}
We stress that this lemma would be immediate if $\Phi$ were assumed to be a \emph{strictly increasing} continuously differentiable function 
on $(0,\infty)$. However, the $\Phi$ under study here is strictly decreasing on $(0,t_o)$ and strictly increasing on $(t_o,\infty)$, why our analysis 
below will need to handle these intervals separately. 
\end{remark}

\begin{proof}[Proof of Proposition \ref{Prop_Phi-inequality} assuming Lemma \ref{Lemma_EntAlternative}]
Pick $f \in \cP_\Phi$ and $0 < \delta < t_o$. We use Lemma \ref{Lemma_EntAlternative} to write
\[
\Ent_\xi^\Phi(f) = \int_0^\delta \xi(\{ f \geq u\}) \, \Phi'(u) \, du
+
\int_\delta^\infty \xi(\{ f \geq u\}) \, \Phi'(u) \, du.
\]
Let us consider the second term on the right hand side. Since $f \in \cP_\Phi$, we have 
\[
\lim_{t \ra \infty} \alpha_f(t) \Phi'(t) = 0,
\] 
and since $t \mapsto \alpha_f(t)$ is decreasing, it is differentiable almost everywhere, with derivative 
\[
\alpha_f'(t) = - \xi(\{f \geq t\}), \quad \textrm{for Lebesgue almost every $t \in (\delta,\infty)$}.
\]
Partial integration now yields,
\begin{eqnarray*}
\int_\delta^\infty \xi(\{ f \geq u\}) \, \Phi'(u) \, du 
&=& 
\Big[ -\alpha_f(u) \Phi'(u) \Big]_\delta^\infty + \int_\delta^\infty \alpha_f(t) \Phi''(t) \, dt \\
&=&
\alpha_f(\delta) \Phi'(\delta) + \int_\delta^\infty \alpha_f(t) \Phi''(t) \, dt \\
&=&
(\alpha_f(\delta) - 1) \Phi'(\delta) + \Phi'(\delta) + \int_\delta^\infty \alpha_f(t) \Phi''(t) \, dt,
\end{eqnarray*}
and thus
\begin{equation}
\label{Ent}
\Ent_\xi^\Phi(f) - \int_\delta^\infty \alpha_f(t) \Phi''(t) \, dt - \Phi'(\delta) = \int_0^\delta \xi(\{ f \geq u\}) \, \Phi'(u) \, du + (\alpha_f(\delta) - 1) \Phi'(\delta).
\end{equation}
Let us now estimate the right hand side above. Since $\Phi(0) = 0$ and $\Phi'(u) < 0$ for all $u \in (0,t_o)$, and since 
$\xi(\{f \leq u\}) \leq 1$ for all $u$, we have
\[
\Big| \int_0^\delta \xi(\{ f \geq u\}) \, \Phi'(u) \, du \Big| \leq -\int_0^\delta \Phi'(u) \, du \leq -\Phi(\delta) + \Phi(0) = -\Phi(\delta),
\]
as well as,
\[
\big| \alpha_f(\delta) - 1\big| = \big|1 - \int_0^\delta \xi(\{f \geq u\}) \, du - 1\big| \leq \delta.  
\]
Since both $\Phi$ and $\Phi'$ are negative on $(0,t_o)$, this shows that the right hand side in \eqref{Ent} is bounded above in 
absolute value by $-2\max(\Phi(\delta),\delta \, \Phi'(\delta))$, which finishes the proof of Proposition \ref{Prop_Phi-inequality}.
\end{proof}

\subsection{Proof of Lemma \ref{Lemma_EntAlternative}}

To prove Lemma \ref{Lemma_EntAlternative}, we shall make use of the following standard identity, which follows from Fubini's Theorem.
Let $\theta$ be a non-negative measure on $X$, and let $h$ be a non-negative measurable function on $X$. Then, 
\begin{equation}
\label{eq_trivialidentity0}
\int_X h(x) \, d\theta(x) 
=
\int_X \underbrace{\Big( \int_0^\infty \chi_{\{h \, \geq \, \tau\}}(x) \, d\tau \Big)}_{{= \, h(x)}} \, d\theta(x) = \int_0^\infty \theta(\{h \geq \tau \}) \, d\tau.
\end{equation}
Pick $f \in \cP_\Phi$ and define the sub-probability measures
\[
\xi_f^{-} = \xi( \cdot \cap \{f \, < \, t_o\}) \qand \xi_f^{+} = \xi( \cdot \cap \{f \, \geq \, t_o\}), \quad \textrm{on $\cB_X$}.
\]
We set $\Phi_o = \Phi - \Phi(t_o)$, and note that $\Phi_o \geq 0$ on $[0,\infty)$, and
\begin{eqnarray}
\int_X \Phi(f) \, d\xi 
&=& 
\int_X \Phi_o(f) \, d\xi + \Phi(t_o) = \int_{\{ f < \, t_o\}} \Phi_o(f) \, d\xi + \int_{\{ f \geq \, t_o\}} \Phi_o(f) \, d\xi + \Phi(t_o)  \nonumber \\ 
&=&
\int_{X} \Phi_o(f) \, d\xi_f^{-} + \int_{X} \Phi_o(f) \, d\xi_f^{+} + \Phi(t_o). \label{Phidecomp}
\end{eqnarray}
The identity \eqref{eq_trivialidentity0} applied to $\theta = \xi_f^{\pm}$ now allows us to write 
\[
\int_{X} \Phi_o(f) \, d\xi_f^{-} = \int_0^\infty \xi_f^{-}\big( \{ \Phi_o(f) \geq \tau \} \big) \, d\tau
\]
and
\[
\int_{X} \Phi_o(f) \, d\xi_f^{+} = \int_0^\infty \xi_f^{+}\big( \{ \Phi_o(f) \geq \tau \} \big) \, d\tau.
\]
Let us begin by analysing the integral against $\xi_f^{-}$. Since $\Phi_o$ is decreasing on $[0,t_o]$, we see that
\[
\Phi_o(0) \geq \Phi_o(f) \geq \Phi_o(t_o) = 0 \quad \textrm{on the set $\{ f < t_o\}$},
\]
whence
\[
\int_0^\infty \xi_f^{-}\big( \{ \Phi_o(f) \geq \tau \} \big) \, d\tau = \int_0^{\Phi_o(0)} \xi_f^{-}\big( \{ \Phi_o(f) \geq \tau \} \big) \, d\tau.
\]
We make the variable substitution $\tau = \Phi_o(t)$ so that $t$ runs from $t_o$ to $0$ and $d\tau = \Phi'(t) \, dt$. Hence,
\[
\int_0^{\Phi_o(0)} \xi_f^{-}\big( \{ \Phi_o(f) \geq \tau \} \big) \, d\tau = \int_{t_o}^{0} \xi_f^{-}\big( \{ \Phi_o(f) \geq \Phi_o(t) \} \big) \, \Phi'(t) \, dt
\]
The assumption that $\Phi_o$ is decreasing on $[0,t_o]$ implies that 
\[
\{ f \leq t \} = \{\Phi_o(f) \geq \Phi_o(t) \big\} \cap \{ f < t_o\}, \quad \textrm{for all $0 \leq t \leq t_o$}, 
\]
whence 
\begin{eqnarray*}
\int_{t_o}^{0} \xi_f^{-}\big( \{ \Phi_o(f) \geq \Phi_o(t) \} \big) \, \Phi'(t) \, dt 
&=& 
\int_{t_o}^{0} \xi_f^{-}\big( \{ f \leq t \} \big) \, \Phi'(t) \, dt = \int_{t_o}^{0} \xi\big( \{ f \leq t \} \big) \, \Phi'(t) \, dt \\
&=& -\int_0^{t_o} (1 - \xi(\{ f > t\}) \, \Phi'(t) \, dt \\
&=& 
-\Phi(t_o) + \int_0^{t_o} \xi(\{f > t\}) \, \Phi'(t) \, dt,
\end{eqnarray*}
where we in the last identity have used our assumption that $\Phi(0) = 0$. Since the map $t \mapsto \xi(\{ f > t \})$ is monotone decreasing, it has at most countably many discontinuities. In 
particular, we have $\xi(\{ f > t \}) = \xi(\{f \geq t\})$ for Lebesgue almost every $t$, and thus
\[
\int_0^{t_o} \xi(\{f > t\}) \, \Phi'(t) \, dt = \int_0^{t_o} \xi(\{f \geq t\}) \, \Phi'(t) \, dt,
\]
from which we conclude that
\[
\int_{X} \Phi_o(f) \, d\xi_f^{-} = -\Phi(t_o) + \int_0^{t_o} \xi(\{f \geq t\}) \, \Phi'(t) \, dt.
\]
Let us now turn to the $\xi_f^{+}$-integral above. Since $\Phi_o$ is increasing on $[t_o,\infty)$, we can make the variable substitution 
$\tau = \Phi_o(t)$, so that $t$ runs from $t_o$ to $\infty$ and $d\tau = \Phi'(t) \, dt$. Furthermore, we have
\[
\{ f \geq t \} = \{\Phi_o(f) \geq \Phi_o(t) \big\} \cap \{ f \geq t_o\}, \quad \textrm{for all $t \geq t_o$},
\]
whence 
\begin{eqnarray*}
\int_0^\infty \xi_f^{+}\big( \{ \Phi_o(f) \geq \tau \} \big) \, d\tau
&=&
\int_{t_o}^\infty \xi_f^{+}\big( \{ \Phi_o(f) \geq \Phi_o(t) \} \big) \, \Phi'(t) \, dt \\
&=&
\int_{t_o}^\infty \xi\big( \{ f \geq t \} \big) \, \Phi'(t) \, dt 
\end{eqnarray*}
We now conclude that
\begin{eqnarray*}
\int_X \Phi(f) \, d\xi 
&=& 
\int_{X} \Phi_o(f) \, d\xi_f^{-} + \int_{X} \Phi_o(f) \, d\xi_f^{+} + \Phi(t_o) \\[3pt]
&=& 
-\Phi(t_o) + \int_0^{t_o} \xi(\{f \geq t\}) \, \Phi'(t) \, dt + \int_{t_o}^\infty \xi\big( \{ f \geq t \} \big) \, \Phi'(t) \, dt + \Phi(t_o) \\[3pt]
&=&
\int_0^\infty \xi\big( \{ f \geq t \} \big) \, \Phi'(t) \, dt,
\end{eqnarray*}
which finishes the proof.

\section{Proof of Theorem \ref{thm_entropycont}}
\label{sec:proofofentropycont}

We retain the assumptions on $\Phi$ from Theorem \ref{thm_upper/lower-limits}. Throughout the rest of this section, we fix 
$\rho \in \cP_\Phi$, a sequence $(\cA_n)$ of sub-$\sigma$-algebras of $\cB_X$ and a sub-$\sigma$-algebra 
$\cA \subset \cB_X$. We set
\[
f_n = \bE_{\cA_n}(\rho) \qand f = \bE_{\cA}(\rho). 
\]
To deduce Theorem \ref{thm_entropycont} from Theorem \ref{thm_upper/lower-limits}, we need to prove:
\vspace{0.1cm}
\begin{enumerate}
\item[(1)] $f_n \in \cP_\Phi$ for all $n$, and $f \in \cP_\Phi$.
\vspace{0.1cm}
\item[(2)] If $\cA$ is a lower Kud\textoverline{o}-limit of $(\cA_n)$, then $f$ is a lower limit of the sequence $(f_n)$.
\vspace{0.1cm}
\item[(3)] If $\cA$ is an upper Kud\textoverline{o}-limit of $(\cA_n)$, then $f$ is a upper limit of the sequence $(f_n)$.
\vspace{0.1cm}
\item[(4)] If $\int_X \rho^{1+\beta} \, d\xi$ for some $\beta > 0$, then $\sup_n \int_X f_n^{1+\beta} \, d\xi < \infty$.
\end{enumerate}
\vspace{0.1cm}
We note that (4) is immediate from Jensen's inequality for conditional expectations. The other points are consequences of the 
following simple lemma (see details below):

\begin{lemma}
\label{Lemma_rho}
Suppose that $f$ is a non-negative $\xi$-integrable function on $X$. Then, 
\[
\int_X |f(x)-t| \, d\xi(x) = 2 \alpha_f(t) - \int_X f(x) \, d\xi(x) + t, \quad \textrm{for every $t \geq 0$.}
\]
\end{lemma}

\begin{proof}
Fix $t \geq 0$ and note that
\begin{eqnarray}
\int_X |f(x) - t | \, d\xi(x)
&=& 
\int_{\{f \geq t\}} (f(x)-t) \, d\xi(y) - \int_{\{f < t\}} (f(x)-t) \, d\xi(y) \nonumber \\
&=&
2 \int_{\{f \geq t\}} (f(x)-t) \, d\xi(x) - \int_{X} (f(x)-t) \, d\xi(y) \nonumber \\
&=&
2 \int_{\{f \geq t\}} f(x) \, d\xi(x) - 2t \xi(\{ f \geq t\}) - \int_X f(x) \, d\xi(x) + t. \label{f-t}
\end{eqnarray}
Let $h_t = f \chi_{\{f \geq t\}}$, so that 
\[
\xi(\{h_t \geq \tau\}) 
= 
\left\{
\begin{array}{cc}
\xi(\{ f \geq t \}) & \textrm{if $t \geq \tau$} \\
& \\
\xi(\{ f \geq \tau \}) & \textrm{if $t < \tau$} 
\end{array}
\right., \quad \textrm{for all $\tau \geq 0$}.
\]
We recall that 
\begin{equation}
\label{eq_trivialidentity}
\int_X f(x) \, d\xi(x) 
=
\int_X \underbrace{\Big( \int_0^\infty \chi_{\{f \, \geq \, t\}}(x) \, dt \Big)}_{{= \, f(x)}} \, d\xi(x) = \int_0^\infty \xi(\{f \geq t \}) \, dt,
\end{equation}
whence
\begin{eqnarray*}
\int_{\{f \geq t\}} f(x) \, d\xi(x) 
&=& \int_X h_t(x) \, d\xi(x) = \int_0^\infty \xi(\{h_t \geq \tau\}) \, d\tau \\
&=& \int_0^t \xi(\{h_t \geq \tau\}) \, d\tau + \int_t^\infty \xi(\{h_t \geq \tau\}) \, d\tau \\
&=& \int_0^t \xi(\{f \geq t\}) \, d\tau + \int_t^\infty \xi(\{f \geq \tau\}) \, d\tau \\ [10pt]
&=& t \xi(\{ f \geq t\}) + \alpha_f(t).
\end{eqnarray*}
If we plug this into \eqref{f-t}, we conclude that
\[
\int_X |f(x) - t | \, d\xi(x) = 2 \alpha_f(t) - \int_X f(x) \, d\xi(x) + t.
\]
\end{proof}

The following result is an immediate corollary of Lemma \ref{Lemma_rho}.

\begin{corollary}
\label{cor_lowupplimit}
Let $(f_n)$ be a sequence of non-negative $\xi$-integrable functions on $X$ and let $f$ be a non-negative $\xi$-integrable
function on $X$. Suppose that $\int_X f_n \, d\xi = \int_X f \, d\xi$ for all $n$. Then:
\begin{enumerate}
\item[(i)] $f$ is a lower limit of $(f_n)$ if and only if
\[
\int_X |f(x) - t | \, d\xi(x) \leq \varliminf_n \int_X |f_n - t | \, d\xi(x), \quad \textrm{for all $t \geq 0$}.
\]
\item[(ii)] $f$ is an upper limit of $(f_n)$ if and only if
\[
\varlimsup_n \int_X |f_n(x) - t | \, d\xi(x) \leq \int_X |f(x) - t | \, d\xi(x), \quad \textrm{for all $t \geq 0$}.
\]
\end{enumerate}
\end{corollary}

\subsection{Proof of (1)}

We recall that $\cP_\Phi$ denotes the set of all measurable functions $f : X \ra [0,\infty)$ such that
\[
\int_X f(x) \, d\xi(x) = 1 \qand \int_X |\Phi(f(x))| \, d\xi(x) < \infty \qand \lim_{t \ra \infty} \alpha_f(t) \Phi'(t) = 0.
\]
So the first condition is fulfilled for 
$\bE_{\cA_n}(\rho)$ and $\bE_{\cA}(\rho)$, as $\rho$ belongs to $\cP_\Phi$.Since $\Phi$ is assumed to be convex and bounded from below, 
the second condition essentially follows from Jensen's inequality for conditional expectations. To verify the third condition, we need the following
corollary of Lemma \ref{Lemma_rho}. 

\begin{corollary}
\label{Cor_Phireserve}
Let $h$ be a non-negative $\xi$-integrable function on $X$ and let $\cC \subset \cB_X$ be a 
sub-$\sigma$-algebra. Then,
\[
\alpha_{\bE_{\cC}(h)}(t) \leq \alpha_{h}(t), \quad \textrm{for all $t \geq 0$}.
\]
\end{corollary}

\begin{remark}
In particular, this implies that if $\rho \in \cP_\Phi$, then 
\[
\varlimsup_{t \ra \infty} \alpha_{\bE_\cC(\rho)}(t) \, \Phi'(t) \leq \lim_{t \ra \infty} \alpha_{\rho}(t) \, \Phi'(t) = 0,
\]
for any sub-$\sigma$-algebra $\cC \subset \cB_X$. Since we assume that $\Phi'(t) > 0$ for all $t > t_o$, this shows that $\bE_{\cC}(\rho)$
belongs to $\cP_\Phi$ as well, which finishes the proof of (1).
\end{remark}

\begin{proof}[Proof of Corollary \ref{Cor_Phireserve}]
Fix a $\xi$-integrable function $h : X \ra [0,\infty)$ and let $\cC \subset \cB_X$ be a 
sub-$\sigma$-algebra. Then,
\[
\int_{X} \big| \bE_{\cC}(h) - t \big| \, d\xi = \int_{X} \big| \bE_{\cC}(h-t) \big| \, d\xi \leq \int_X |h - t | \, d\xi,
\]
for all $t \geq 0$. Hence, by Lemma \ref{Lemma_rho}, applied to $f = \bE_{\cC}(h)$ and $f = h$ respectively, we have
\[
2 \alpha_{\bE_{\cC}(h)}(t) - \int_X \bE_{\cC}(h) \, d\xi + t \leq 2 \alpha_{h}(t) - \int_X h \, d\xi + t, \quad \textrm{for all $t \geq 0$},
\]
whence $\alpha_{\bE_{\cC}(h)}(t) \leq \alpha_{h}(t)$ for all $t \geq 0$.
\end{proof}

\subsection{Proofs of (2) and (3)}

We recall that a sub-$\sigma$-algebra $\cA \subset \cB_X$ is a lower Kud\textoverline{o}-limit of the sequence $(\cA_n)$ (with respect to $\xi$)
if
\begin{equation}
\label{def2_lowlim}
\|\bE_{\cA}(\psi)\|_{L^1(\xi)} \leq \varliminf_n \| \bE_{\cA_n}(\psi) \|_{L^1(\xi)}, \quad \textrm{for all $\psi \in L^1(X,\xi)$},
\end{equation}
and it is an upper Kud\textoverline{o}-limit of $(\cA_n)$ (with respect to $\xi$) if
\begin{equation}
\label{def2_upplim}
\varlimsup_n \| \bE_{\cA_n}(\psi) \|_{L^1(\xi)} \leq \|\bE_{\cA}(\psi)\|_{L^1(\xi)}, \quad \textrm{for all $\psi \in L^1(X,\xi)$}.
\end{equation}
Lemma \ref{Lemma_rho} allows us to reformulate these notions in the language of lower and upper limits of a sequence 
of functions. The following corollary clearly proves (2) and (3). 

\begin{corollary}
\label{cor_relateKudotoAsymp2nd}
Let $(\cA_n)$ be a sequence of sub-$\sigma$-algebras of $\cB_X$ and let $\cA \subset \cB_X$ be a sub-$\sigma$-algebra.
Then: 
\begin{enumerate}
\item[(i)] $\cA$ is a lower Kud\textoverline{o}-limit of $(\cA_n)$ if and only if for every non-negative $\rho \in L^\infty(X,\xi)$ with 
$\int_X \rho \, d\xi = 1$, the function $\bE_{\cA}(\rho)$ is a lower limit of the sequence $(\bE_{\cA_n}(\rho))$.
\vspace{0.1cm}
\item[(ii)] $\cA$ is an upper Kud\textoverline{o}-limit of $(\cA_n)$ if and only if for every non-negative $\rho \in L^\infty(X,\xi)$ with 
$\int_X \rho \, d\xi = 1$, the function $\bE_{\cA}(\rho)$ is an upper limit of the sequence $(\bE_{\cA_n}(\rho))$.
\end{enumerate}
\end{corollary}

\begin{proof}
The proofs of (i) and (ii) are almost identical, so we only write out the details for (i). Let us first assume that $\cA$ is a lower Kud\textoverline{o}-limit
of $(\cA_n)$ and pick a non-negative $\rho \in L^\infty(X,\xi)$. Fix $t \geq 0$, and set $\psi = \rho - t$. Since $\cA$ is a lower Kud\textoverline{o}-limit
of $(\cA_n)$, we have
\[
\|\bE_{\cA}(\rho) - t\|_{L^1(\xi)} = \|\bE_{\cA}(\psi)\|_{L^1(\xi)} \leq \varliminf_n \| \bE_{\cA_n}(\psi) \|_{L^1(\xi)} 
= \varliminf_n \|\bE_{\cA_n}(\rho) - t\|_{L^1(\xi)},
\]
whence Corollary \ref{cor_lowupplimit}, applied to $g_n = \bE_{\cA_n}(\rho)$ and $g = \bE_{\cA}(\rho)$, shows that $\bE_{\cA}(\rho)$
is a lower limit of the sequence $(\bE_{\cA_n}(\rho))$. This shows the ``only if''-direction. \\

For the ``if''-direction, we first observe that a straightforward approximation argument shows that \eqref{def2_lowlim} holds if and only if it 
holds for all bounded functions $\psi$. Furthermore, \eqref{def2_lowlim} trivially holds for $\psi$ which are constant $\xi$-almost everywhere). 
In what follows, let us fix $\psi \in L^\infty(X,\xi)$, which is not constant $\xi$-almost everywhere, and 
set
\[
\rho = c(\psi + \|\psi\|_\infty),
\]
where $c$ is a strictly positive constant chosen so that $\int_X \rho \, d\xi = 1$. Then $\rho$ is a non-negative $\xi$-integrable function, and by assumption, 
$\bE_{\cA}(\rho)$ is a lower limit of the sequence $(\bE_{\cA_n}(\rho))$. By Corollary \ref{cor_lowupplimit}, this is equivalent to
\[
\int_X |\bE_{\cA}(\rho) - t| \, d\xi \leq \varliminf_n \int_X |\bE_{\cA_n}(\rho) - t| \, d\xi, \quad \textrm{for all $t \geq 0$}.
\]
In particular, if we take $t = c\|\psi\|_\infty$, then this inequality can be rewritten as 
\[
 \int_X |\bE_{\cA}(\psi)| \, d\xi \leq \varliminf_n \int_X |\bE_{\cA_n}(\psi)| \, d\xi,
\]
and thus \eqref{def2_lowlim} holds for $\psi$. Since $\psi$ is arbitrary, we conclude that $\cA$ is a 
lower Kud\textoverline{o}-limit of the sequence $(\cA_n)$, which finishes the proof of the ``if''-direction.
\end{proof}

\section{Proof of Theorem \ref{thm_quantentropycont}}

Throughout this section, we let
\[
\Phi(t) = 
\left\{
\begin{array}{cc}
0 & \textrm{for $t = 0$} \\[5pt]
t \log t & \textrm{for $t > 0$}
\end{array}
\right.
\]
To avoid cluttering, we write
\[
\Ent_\xi(f) = \int_X \Phi(f) \, d\xi = \int_X f \log f \, d\xi, \quad \textrm{for $f \in \cP_\Phi$},
\]
where $\cP_\Phi$ denotes the set of all non-negative measurable functions on $X$ such that 
\[
\int_X f \, d\xi = 1 \qand \int_X |f| \log^{+} f \, d\xi < \infty \qand \lim_{t \ra \infty} \alpha_f(t) \, \log t = 0.
\]
We note that if $f$ is a bounded $\xi$-measurable function on $X$ with $\int_X f \, d\xi = 1$, then all of these conditions
are satisfied. \\

We will deduce Theorem \ref{thm_quantentropycont} from the following two propositions.

\begin{proposition}[Pinsker-Csizs\'ar-Kullback inequality]
\label{Prop_PCK}
Let $(X,\xi)$ be a probability measure space and let $f$ be a non-negative $\xi$-integrable function on $X$ with 
$\int_X f \, d\xi = 1$. Then,
\[
\int_X |1 - f(x) | \, d\xi(x) \leq \sqrt{2} \Ent_\xi(f)^{1/2}.
\]
\end{proposition}
\begin{proposition}
\label{Prop_NormUpper}
Let $(\cA_n)$ be a sequence of sub-$\sigma$-algebras of $\cB_X$ and let $\cA$ be an upper Kud\textoverline{o}-limit of $(\cA_n)$. Then,
for every $\phi \in L^1(X,\xi)$, we have
\[
\lim_n \big\| \bE_{\cA_n}(\phi) - \bE_{\cA_n}(\bE_{\cA}(\phi))\big\|_{L^1(\xi)} = 0.
\]
\end{proposition}
\vspace{0.2cm}
\begin{remark}
A short and elegant proof of the celebrated Pinsker-Csizs\'ar-Kullback inequality (Proposition \ref{Prop_PCK}) can be found in \cite[Subsection 5.2.1]{Bak}.

The proof of Proposition \ref{Prop_NormUpper} is short as well, and can be presented on a few lines as follows: Let $\phi \in L^1(X,\xi)$ and set 
$\psi = \phi - \bE_{\cA}(\phi)$ so that $\bE_{\cA}(\psi) = 0$. Since $\cA$ is an upper Kud\textoverline{o}-limit of the sequence $(\cA_n)$, we have
\[
\varlimsup_n \big\| \bE_{\cA_n}(\phi) - \bE_{\cA_n}(\bE_{\cA}(\phi))\|_{L^1(\xi)}
=
\, \varlimsup_n \big\| \bE_{\cA_n}(\psi)\|_{L^1(\xi)} \leq \|\bE_{\cA}(\psi)\big\|_{L^1(\xi)} = 0.
\]
\end{remark}
\vspace{0.5cm}
Let us now turn to the proof of Theorem \ref{thm_quantentropycont}. We fix $\lambda \geq 1$ and a measurable function 
\[
\rho : X \ra [\lambda^{-1},\lambda], \quad \textrm{with $\int_X \rho \, d\xi = 1$},
\]
as well as a sequence $(\cA_n)$ of sub-$\sigma$-algebras of $\cB_X$. We denote by $\cA^{+}$ and $\cA^{-}$ the maximal lower and minimal
upper Kud\textoverline{o}-limits of the sequence $(\cA_n)$ respectively. We write
\[
\|\bE_{\cA^+}(\rho) - \bE_{\cA_n}(\rho)\|_{L^1(\xi)} = \int_X \Big| 1 - \frac{\bE_{\cA^{+}}(\rho)}{\bE_{\cA_n}(\rho)} \Big| \, \bE_{\cA_n}(\rho) \, d\xi,
\]
and for each $n$, we apply Proposition \ref{Prop_PCK} to 
\[
f = \frac{\bE_{\cA^{+}}(\rho)}{\bE_{\cA_n}(\rho)}  \qand d\mu = \bE_{\cA_n}(\rho) \, d\xi.
\]
We conclude that
\begin{eqnarray*}
\|\bE_{\cA^+}(\rho) - \bE_{\cA_n}(\rho)\|_{L^1(\xi)} 
&\leq &
\sqrt{2} \Ent_\xi(f)^{1/2} \\
&=& \sqrt{2} \, \Big( \Ent_\xi(\bE_{\cA^{+}}(\rho)) - \int_X \bE_{\cA^{+}}(\rho) \, \log \bE_{\cA_n}(\rho) \, d\xi \Big)^{1/2}.
\end{eqnarray*}
We now note that 
\begin{eqnarray*}
\int_X \bE_{\cA^{+}}(\rho) \, \log \bE_{\cA_n}(\rho) \, d\xi 
&=& 
\int_X \bE_{\cA_n}(\bE_{\cA^{+}}(\rho)) \, \log \bE_{\cA_n}(\rho) \, d\xi \\
&=& 
\int_X\bE_{\cA_n}(\rho) \, \log \bE_{\cA_n}(\rho) \, d\xi \\
&+&  
\int_X \big(\bE_{\cA_n}(\bE_{\cA^{+}}(\rho)) - \bE_{\cA_n}(\rho)\big) \, \log \bE_{\cA_n}(\rho) \, d\xi \\
&=&
H_\rho(\cA_n) + \int_X \big(\bE_{\cA_n}(\bE_{\cA^{+}}(\rho)) - \bE_{\cA_n}(\rho)\big) \, \log \bE_{\cA_n}(\rho) \, d\xi.
\end{eqnarray*}
Since $\rho : X \ra [\lambda^{-1},\lambda]$, we see that 
\[
\Big|\int_X \big(\bE_{\cA_n}(\bE_{\cA^{+}}(\rho)) - \bE_{\cA_n}(\rho)\big) \, \log \bE_{\cA_n}(\rho) \, d\xi\Big|
\leq 
\big\| \bE_{\cA_n}(\bE_{\cA^+}(\rho)) - \bE_{\cA_n}(\rho)\big\|_{L^1(\xi)} \, \log \lambda,
\]
which tends to zero as $n \ra \infty$ by Proposition \ref{Prop_NormUpper}. Hence,
\[
\varliminf_n \int_X \bE_{\cA^{+}}(\rho) \, \log \bE_{\cA_n}(\rho) \, d\xi = \varliminf_n H_\rho(\cA_n),
\]
and thus
\begin{eqnarray*}
\varlimsup_n \|\bE_{\cA^+}(\rho) - \bE_{\cA_n}(\rho)\|_{L^1(\xi)} 
&\leq &
\sqrt{2} \, \Big( H_\rho(\cA^{+}) - \varliminf_n \int_X \bE_{\cA^{+}}(\rho) \, \log \bE_{\cA_n}(\rho) \, d\xi \Big)^{1/2} \\
& \leq &
\sqrt{2} \, \Big( H_\rho(\cA^{+}) - \varliminf_n H_\rho(\cA_n) \Big)^{1/2},
\end{eqnarray*}
which finishes the proof of Theorem \ref{thm_quantentropycont}.

\section{Proofs of Theorems ~\ref{thm_Tcont} and ~\ref{thm_kernelconv}}
\label{sec:kernelconv}
Throughout this section, we fix a probability measure space $(T,\eta)$ and a measurable 
kernel function $K : T \times X \ra [0,\infty)$ with the properties:
\vspace{0.2cm}
\begin{enumerate}
\item[(I)] for every $t \in T$, the function $\rho_t = K(t,\cdot)$ on $X$ is measurable, satisfies 
\[
\int_X \rho_t \, d\xi = 1 \qand \int_T \Ent_\xi(\rho_t) \, d\eta(t) < \infty,
\]
and there exists $\lambda_t \geq 1$ such that $\rho_t$ only takes values in $[\lambda_t^{-1},\lambda_t]$. \\
\vspace{0.2cm}
\item[(II)] The bounded linear map $\psi \mapsto f_\psi$ from $L^\infty(T,\eta)$ to $L^1(X,\xi)$, defined by
\[
f_\psi = \int_T \psi(t) \, \rho_t \, d\eta(t), \quad \textrm{for $\psi \in L^\infty(T,\eta)$}
\]
has a norm-dense image $V \subset L^1(X,\xi)$.
\end{enumerate}
\begin{remark}
It suffices to assume in (I) that the $\xi$-essential range of $\rho_t$ is contained $[\lambda_t^{-1},\lambda_t]$, that is to say, 
$\lambda_t := \max(\|\rho_t\|_\infty,\|\rho_t^{-1}\|_\infty) < \infty$ for every $t \in T$.
\end{remark}

\vspace{0.2cm}
Let $(\cA_n)$ be a sequence of sub-$\sigma$-algebras of $\cB_X$. We make the following observations:
\vspace{0.2cm}
\begin{enumerate}
\item[(1)] To prove (i) in Theorem \ref{thm_kernelconv}, we need to show that if $h_\eta(\cA_n) \ra 0$, then
\begin{equation}
\label{limit1}
\lim_n \Big\|\int_X f \, d\xi - \bE_{\cA_n}(f) \Big\|_{L^1(\xi)} = 0, \quad \textrm{for all $f \in V$}.
\end{equation}
Since $V$ is norm-dense in $L^1(X,\xi)$, this proves that $\cA_n \ra \{ \, \emptyset,X\}$ (modulo null sets) in the strong operator topology. 
\vspace{0.2cm}
\item[(2)] To prove (ii) in Theorem \ref{thm_kernelconv}, we need to show that if $\cA^+$ is an upper Kud\textoverline{o}-limit of $(\cA_n)$ and 
$h_\eta(\cA_n) \ra h_\eta(\cA^+)$, then
\begin{equation}
\label{limit2}
\lim_n \big\|\bE_{\cA^+}(f) - \bE_{\cA_n}(f) \big\|_{L^1(\xi)} = 0, \quad \textrm{for all $f \in V$}.
\end{equation}
Since $V$ is norm-dense in $L^1(X,\xi)$, this proves that $\cA_n \ra \cA^+$ (modulo null sets) in the strong operator topology. 
\end{enumerate}
\vspace{0.2cm}
\subsection{Proof of Theorem~\ref{thm_Tcont}}

We need to show that the assumptions above imply  that the conditions of Corollary~\ref{cor_kudocont} are fulfilled. In fact, they are fulfilled for any $\beta>0$.

First, since we consider the standard entropy function, for any
$\beta>0$, we have that $\int_{1}^{\infty}t^{-\beta}\Phi''(t)dt<\infty$. Also, the assumption that  $\rho_t$ is bounded for every $t$, 
implies that for any $\beta>0$, we have that $\int_X \rho_t^{1+\beta} d\xi <\infty $. Hence by Corollary~\ref{cor_kudocont} we get convergence for each $t\in T$.

Since we assumed in (I) that the $\eta$-integral of the entorpies of the $\rho_t$ is finite, the Bounded Convergence Theorem now concludes the proof.

\subsection{Proof of (i) in Theorem \ref{thm_kernelconv}}

Let us pick $\psi \in L^\infty(T,\eta)$, and consider $f_\psi \in V$. We have
\begin{equation}
\label{ineq1}
\Big\|\int_X f_\psi \, d\xi - \bE_{\cA_n}(f_\psi) \Big\|_{L^1(\xi)} \leq \|\psi\|_\infty \, \int_T \big\|1 - \bE_{\cA_n}(\rho_t) \big\|_{L^1(\xi)} \, d\eta(t), \quad \textrm{for all $n$}.
\end{equation}
By the Pinsker-Csizs\'ar-Kullback inequality (Proposition \ref{Prop_PCK}), applied to the functions 
$f = \bE_{\cA_n}(\rho_t)$, we see that
\begin{eqnarray*}
\Big\|\int_X f_\psi \, d\xi - \bE_{\cA_n}(f_\psi) \Big\|_{L^1(\xi)} 
&\leq & 
\sqrt{2} \|\psi\|_\infty \int_T H_{\rho_t}(\cA_n)^{1/2} \, d\eta(t) \\
&\leq &
\sqrt{2} \|\psi\|_\infty \Big( \int_T H_{\rho_t}(\cA_n) \, d\eta(t) \Big)^{1/2} \\[4pt]
&=& 
\sqrt{2} \|\psi\|_\infty \, h_\eta(\cA_n)^{1/2},
\end{eqnarray*}
where we have used H\"older's inequality in the second inequality. Hence, if $h_\eta(\cA_n) \ra 0$, then 
\[
\lim_n \Big\|\int_X f_\psi \, d\xi - \bE_{\cA_n}(f_\psi) \Big\|_{L^1(\xi)} = 0.
\]
Since $f_\psi \in V$ is arbitrary, we have established \eqref{limit1}, and thus (i) in Theorem \ref{thm_kernelconv}, in view of (1).

\subsection{Proof of (ii) in Theorem \ref{thm_kernelconv}}

Let us assume that $h_\eta(\cA_{n}) \ra h_\eta(\cA^+)$. We shall show that for every sub-sequence $(n_k)$, there is a further sub-sequence $(n_{k_j})$ such that 
\begin{equation}
\label{reduc}
\lim_j \big\|\bE_{\cA^+}(f) - \bE_{\cA_{n_{k_j}}}(f) \big\|_{L^1(\xi)} = 0, \quad \textrm{for all $f \in V$}.
\end{equation}
In view of (2), this establishes (ii) in Theorem \ref{thm_kernelconv}. We will need the following lemma, whose proof is presented in the next sub-section.

\begin{lemma}
\label{Lemma_subseq}
If $h_\eta(\cA_m) \ra h_\eta(\cA^+)$, then there exists a sub-sequence $(m_l)$ such that
\[
H_{\rho_t}(\cA_{m_l}) \ra H_{\rho_t}(\cA^+), \quad \textrm{for $\eta$-almost every $t \in T$}.
\]
\end{lemma}

Let us now pick a sub-sequence $(n_k)$. Since $\cA^+$ is also an upper Kud\textoverline{o}-limit of $(\cA_{n_k})$ and since $\lim_k h_\eta(\cA_{n_k}) \ra h_\eta(\cA^+)$, Lemma \ref{Lemma_subseq} allows us to 
extract a further sub-sequence $(n_{k_j})$ such that 
\begin{equation}
\label{almostsure}
\lim_j H_{\rho_t}(\cA_{n_{k_j}}) = H_{\rho_t}(\cA^+), \quad \textrm{for $\eta$-almost every $t \in T$}.
\end{equation}
Let us now pick $\psi \in L^\infty(T,\eta)$, and consider $f_\psi \in V$. We have
\begin{equation}
\label{ineq2}
\big\|\bE_{\cA^+}(f_\psi) - \bE_{\cA_{n_{k_j}}}(f_\psi) \big\|_{L^1(\xi)}  \leq \|\psi\|_\infty \, \int_T \big\| \bE_{\cA^+}(\rho_t) - \bE_{\cA_{n_{k_j}}}(\rho_t) \big\|_{L^1(\xi)} \, d\eta(t),
\end{equation}
for all $j$, so to prove \eqref{reduc}, we need to show that
\begin{equation}
\label{reduc2}
\lim_{j} \int_T \big\| \bE_{\cA^+}(\rho_t) - \bE_{\cA_{n_{k_j}}}(\rho_t) \big\|_{L^1(\xi)} \, d\eta(t) = 0.
\end{equation}
Since $\sup_j \big\| \bE_{\cA^+}(\rho_t) - \bE_{\cA_{n_{k_j}}}(\rho_t) \big\|_{L^1(\xi)} \leq 2$ for all $t \in T$, the limsup-version of Fatou's Lemma can be applied, so \eqref{reduc2} follows if we can prove: 
\begin{equation}
\label{reduc3}
\int_T \varlimsup_j \big\| \bE_{\cA^+}(\rho_t) - \bE_{\cA_{n_{k_j}}}(\rho_t) \big\|_{L^1(\xi)} \, d\eta(t) = 0.
\end{equation}
By Theorem \ref{thm_quantentropycont}, applied to $\rho = \rho_t$ for every $t \in T$, and the $\eta$-almost sure limit 
\eqref{almostsure}, we have
\[
\int_T \varlimsup_j \big\| \bE_{\cA^+}(\rho_t) - \bE_{\cA_{n_{k_j}}}(\rho_t) \big\|_{L^1(\xi)} \, d\eta(t)
\leq 
\sqrt{2} \int_T \Big( H_{\rho_t}(\cA^+) - \varliminf_j H_{\rho_t}(\cA_{n_{k_j}}) \Big)^{1/2} \, d\eta(t) = 0,
\]
which proves \eqref{reduc3}, and thus (ii) in Theorem \ref{thm_kernelconv}.

\subsection{Proof of Lemma \ref{Lemma_subseq}}

We shall show that if $h_\eta(\cA_m) \ra h_\eta(\cA^+)$, then 
\begin{equation}
\label{normconv}
\lim_m \int_T \big| H_{\rho_t}(\cA^{+}) - H_{\rho_t}(\cA_m)\big| \, d\eta(t) = 0,
\end{equation}
whence, by a standard Borel-Cantelli argument, we can extract at least one sub-sequence $(m_l)$ such that 
$H_{\rho_t}(\cA_{m_l}) \ra H_{\rho_t}(\cA^+)$ for $\eta$-almost every $t \in T$. To prove \eqref{normconv}, we set
\[
\varphi_m(t) = H_{\rho_t}(\cA_m) \qand \varphi(t) = H_{\rho_t}(\cA^{+}).
\]
By Jensen's inequality for conditional expectations, we have for every $t \in T$ and for every sub-$\sigma$-algebra $\cC \subset \cB_X$,
\begin{equation}
\label{Jensen}
H_{\rho_t}(\cC) \leq \Ent_\xi(\rho_t).
\end{equation}
We now know that:
\vspace{0.2cm}
\begin{enumerate}
\item[(A)] both $\sup_m \varphi_m$ and $\varphi$ are $\eta$-integrable (by \eqref{Jensen} and (I)). 
\vspace{0.2cm}
\item[(B)] $\varlimsup_m \varphi_m(t) \leq \varphi(t)$ for every $t \in T$ (by (ii) in Theorem \ref{thm_entropycont}). 
\vspace{0.2cm}
\item[(C)] $\lim_m \int_T \varphi_m \, d\eta = \int_T \varphi \, d\eta$ (by our assumption that $h_\eta(\cA_m) \ra h_\eta(\cA^+)$).
\end{enumerate}
\vspace{0.2cm}
We claim that these properties force $\lim_m \int_T |\varphi - \varphi_m | \, d\eta = 0$. To see this, we write
\begin{eqnarray*}
\int_T |\varphi - \varphi_m| \, d\eta 
&=& 
- \int_{\{\varphi \, < \, \varphi_m\}} (\varphi - \varphi_m) \, d\eta + \int_{\{ \varphi \, \geq \, \varphi_m\}} (\varphi - \varphi_m) \, d\eta \\
&=&
- \int_{\{\varphi \, <  \, \varphi_m\}} (\varphi - \varphi_m) \, d\eta + \int_{T} (\varphi - \varphi_m) \, d\eta \\
&-& 
\int_{\{ \varphi \, < \, \varphi_m\}} (\varphi - \varphi_m) \, d\eta \\
&=&
\int_T \varphi \, d\eta - \int_T \varphi_m \, d\eta - 2 \, \int_{\{ \varphi \, < \, \varphi_m\}} (\varphi - \varphi_m) \, d\eta.
\end{eqnarray*}
By Condition (C), the difference of the first two terms in the last identity tend to zero as $m \ra \infty$, so to prove \eqref{normconv}, it 
suffices to show that the third term tends to zero as well. This will clearly follow if we can prove:
\begin{equation}
\label{F}
\varlimsup_m \int_{\{ \varphi \, < \, \varphi_m\}} F \, d\eta = 0,
\end{equation}
where $F = \sup_m \varphi_m + \varphi$. By (A), $F$ is an $\eta$-integrable non-negative function on $T$, so by the limsup-version of 
Fatou's Lemma, we have
\[
\varlimsup_m \int_{\{ \varphi \, < \, \varphi_m\}} F \, d\eta \leq \int_{T} F \, \varlimsup_m \chi_{ \{ \varphi \, < \, \varphi_m\}} \, d\eta.
\]
By Condition $(B)$, we see that $\varlimsup_m \chi_{ \{ \varphi \, < \, \varphi_m\}} = 0$ $\eta$-almost surely, and thus \eqref{F} follows. This finishes
the proof of Lemma \ref{Lemma_subseq}.

\section{Proofs of Theorems \ref{thm_bndcont} and \ref{thm_FurstenbergEntropy}}
\label{sec:FEnt}
Let $(G,\mu)$ be a measured group and let $(B,\nu)$ denote the Poisson boundary of $(G,\mu)$. We recall from the introduction that if 
$(Z,\theta)$ is a Borel $(G,\mu)$-space, then its \emph{Furstenberg $\mu$-entropy} $h_{(G,\mu)}(Z,\theta)$ is defined as
\[
h_{(G,\mu)}(Z,\theta) = \int_G \int_Z -\log \rho_{g^{-1}}^\theta(z) \, d\theta(z) \, d\mu(g).
\] 
We refer to Section \ref{sec:Borelprel}\ for further definitions and standing assumptions. \\

In what follows, we fix a sequence $((Z_n,\theta_n))$ of $\mu$-boundaries, along with measurable and $G$-equivariant maps 
\[
\pi_n : (B,\nu) \ra (Z_n,\theta_n), \quad \textrm{for all $n$}.
\]
Consider the sequence $(\cA_n)$ in $\frak{S}(B,\nu)$ defined by $\cA_n = \pi_n^{-1}(\cB_{Z_n})$ and denote by $\cA^+$ and $\cA^-$ the 
minimal upper and maximal lower Kud\textoverline{o}-limits of $(\cA_n)$ respectively. As we point out in Subsection \ref{subsec:BorelG}, 
there are $\mu$-boundaries $(Z^+,\theta^+)$ and $(Z^-,\theta^-)$ and measurable $G$-equivariant maps 
\[
\pi_+ : (B,\nu) \ra (Z^+,\theta^+) \qand \pi_- : (B,\nu) \ra (Z^-,\theta^-) 
\]
such that $\pi_{+}^{-1}(\cB_{Z^+}) = \cA^+$ and $\pi_{-}^{-1}(\cB_{Z^-}) = \cA^-$ modulo $\xi$-null sets. \\

In this language, Theorem \ref{thm_FurstenbergEntropy} amounts to proving:
\vspace{0.2cm}
\begin{enumerate}
\item[(1)] If $h_{(G,\mu)}(Z_n,\theta_n) \ra 0$, then $(\cA_n)$ converges to $\{\emptyset,B\}$ in $\frak{S}(B,\nu)$.
\vspace{0.2cm}
\item[(2)] If $h_{(G,\mu)}(Z_n,\theta_n) \ra h_{(G,\mu)}(Z^+,\theta^+)$, then $\cA_n \ra \cA^+$ in $\frak{S}(B,\nu)$.
\end{enumerate}
These assertions are in fact special cases of Theorem \ref{thm_kernelconv}. To see this, we first note that
Lemma \ref{lemma_entropy} allows us to write
\[
h_{(G,\mu)}(Z_n,\theta_n) = \gamma^{-1} \, \int_G \Ent_\nu(\bE_{\cA_n}(\rho_g^{\nu})) \, d\eta_\mu(g) =: \gamma^{-1} \, h_{\eta_\mu}(\cA_n),
\]
for every $n$, and
\[
h_{(G,\mu)}(Z^+,\theta^+) = \gamma^{-1} \, \int_G \Ent_\nu(\bE_{\cA^+}(\rho_g^{\nu})) \, d\eta_\mu(g) =: \gamma^{-1} \, h_{\eta_\mu}(\cA^{+}),
\]
where $\gamma = \sum_{k=1}^\infty \frac{k}{2^{k}}$, and thus (1) and (2) follow from Theorem \ref{thm_kernelconv} (i) and (ii) respectively, applied to 
\[
(T,\eta) = (G,\eta_\mu) \qand \rho_g = \rho_g^\nu,
\]
\emph{provided} that we can show that the Radon-Nikodym kernel $K(g,y) = \frac{dg\nu}{d\nu}(y) = \rho^{\nu}_g(y)$ on $G \times B$ satisfies the conditions of that theorem.  In other words, 
we need to show that:
\vspace{0.2cm}
\begin{enumerate}
\item[a)] for every $g \in G$, the function $\rho^{\nu}_g$ on $B$ is measurable, satisfies 
\[
\int_B \rho^{\nu}_g \, d\nu = 1 \qand \int_G \Ent_\nu(\rho^{\nu}_g) \, d\eta_\mu(g) < \infty,
\]
and there exists $\lambda_g \geq 1$ such that the $\nu$-essential range of $\rho^{\nu}_g$ is contained in $[\lambda_g^{-1},\lambda_g]$.
\vspace{0.2cm}
\item[b)]
The bounded linear map $\psi \mapsto f_\psi$ from $L^\infty(G,\eta_\mu)$ to $L^1(B,\nu)$, defined by
\[
f_\psi = \int_G \psi(g) \, \rho^{\nu}_g \, d\eta_\mu(g), \quad \textrm{for $\psi \in L^\infty(G,\eta_\mu)$}
\]
has a norm-dense image $V \subset L^1(B,\nu)$.
\end{enumerate}
\vspace{0.1cm}
Concerning a), we note that the first two assertions are contained in Lemma \ref{lemma_BorelG}, while the final assertion is
contained in Lemma \ref{lemma_RNbounded}. Furthermore, by (i)
Lemma \ref{lemma_entropy}, the integral $\int_G \Ent_\nu(\rho^{\nu}_g) \, d\eta_\mu(g)$ is finite if $h_{(G,\mu)}(B,\nu)$ is
finite, which readily follows from local boundedness of the map $g \mapsto \lambda_g$ in Lemma \ref{lemma_RNbounded}
and the relatively compactness of  $u$. Finally, we note that b) is established in Lemma \ref{lemma_poissonmap}, and thus the proof
of Theorem \ref{thm_FurstenbergEntropy} is complete.
\\

The proof of Theorem~\ref{thm_bndcont} now follows from Theorem~\ref{thm_Tcont} since we have seen above that the Radon-Nikodym derivatives fulfill the criteria of being kernels in the sense of Section~\ref{subsec:av}.

\end{document}